\documentclass[psamsfonts,reqno]{amsart}
\usepackage[utf8]{inputenc}
\usepackage{amsfonts}
\usepackage{hyperref}
\usepackage[capitalize]{cleveref}
\usepackage{amsmath}
\usepackage{xcolor}
\usepackage{amsthm}
\usepackage{bbm}
\usepackage{pdflscape}
\usepackage{pgfplots}
\usepackage{mathrsfs}
\usepackage{enumerate}
\usepackage[backend=biber, style=numeric]{biblatex}
\newtheorem{thm}{Theorem}[section]
\newtheorem{cor}[thm]{Corollary}
\newtheorem{prop}[thm]{Proposition}
\newtheorem{lem}[thm]{Lemma}
\newtheorem{conj}[thm]{Conjecture}
\newtheorem{quest}[thm]{Question}

\theoremstyle{definition}
\newtheorem{defn}[thm]{Definition}

\newtheorem{exmp}[thm]{Example}

\newtheorem{prob}[thm]{Problem}

\numberwithin{equation}{section}

\theoremstyle{remark}
\newtheorem{rem}[thm]{Remark}

\renewcommand{\qed}{\unskip\nobreak\quad\qedsymbol}
\usepackage[margin=1.53in]{geometry}
\newcommand{\setone}{\mathbbm{1}}
\newcommand{\setA}{\mathbb{A}}
\newcommand{\setT}{{\mathbb{T}}}

\newcommand{\setR}{{\mathbb{R}}}

\newcommand{\setZ}{{\mathbb{Z}}}
\newcommand{\setN}{{\mathbb{N}}}

\allowdisplaybreaks

\begin{document}

\title{Fluctuation of Ergodic Averages and Other Stochastic Processes}

\author{S. Mondal}
\address{Department of Mathematics, The Ohio State University\\
Columbus, OH 43210, USA\\}
\email{mondal.56@osu.edu}

\author{J. Rosenblatt}
\address{Department of Mathematics\\
University of Illinois at Urbana-Champaign\\
Urbana, IL 61801}
\email{rosnbltt@illinois.edu}

\author{M. Wierdl}
\address{Department of Mathematical Sciences\\
University of Memphis\\
Memphis, TN 38152}
\email{mwierdl@memphis.edu}

\begin{abstract}  For an ergodic map $T$ and a non-constant, real-valued $f\in L^1$, the ergodic averages $A_Nf(x) =\frac 1N\sum\limits_{n=1}^N f (T^n x)$ converge a.e., but  the convergence is never monotone.  Depending on particular properties of the function $f$, the averages $A_Nf(x)$ may or may not actually fluctuate around the mean value infinitely often a.e.

We will prove that a.e. fluctuation around the mean is the generic behavior.  That is, for a fixed ergodic $T$, the generic non-constant $f\in L^1$ has the averages $A_Nf(x)$ fluctuating around the mean infinitely often for almost every $x$.

We also consider fluctuation for other stochastic processes like subsequences of the ergodic averages, convolution operators, weighted averages, uniform distribution and martingales.  We will show that in general, in these settings fluctuation around the limit infinitely often persists as the generic behavior.
\end{abstract}
\date{}
\maketitle

\tableofcontents

\section{Introduction}\label{Intro}

The aim of this paper is to analyze the non-monotonicity and the fluctuation of a variety of convergent stochastic processes. Suppose that $(X,\mathcal{B},\mu)$ is a standard probability space, $T_N: L^1 \to L^1 $ is a sequence of bounded, linear operators such that for every real-valued, integrable function $f$, $(T_Nf)$ converges a.e. to $T f$ for some bounded linear operator\linebreak $T:L^1\to L^1$. 
\begin{defn}\label{non-monotone}
With this setup, we say that 
the convergence of $T_N f$ is \emph{non-monotone} if there does not exist any set $M\subset X$ with $\mu(M) > 0$ satisfying the following:\\ For every $x\in M$, $T_Nf(x)\geq T_{N+1}f(x)$ eventually, or $T_Nf(x)\leq T_{N+1}f(x)$ eventually.
\end{defn}

Given a sequence of bounded linear operators and its limit, we are interested in the following questions.

\begin{quest}\label{ques:1.1}[Non-monotonicity] 
For any function $f\in L^1$, is the convergence of $T_N f$ {non-monotone}?
\end{quest}

When this question has an affirmative answer, we ask a more pertinent question about the behavior of convergence. Namely,
\begin{quest}\label{ques:1.2}[Fluctuation] 
 Can we say that for almost every $x\in X$, $T_N f(x)$ \emph{fluctuates} around   its limit infinitely often, that is, is $(T_N f(x)- T f(x))>0$ for infinitely many $N$, and $(T_N f(x)- T f(x))<0$ for infinitely many (different) $N$? 
\end{quest}

The study of non-monotonicity and fluctuation of convergent stochastic processes is a topic of intrinsic interest in probability theory. For a symmetric random walk, non-monotonicity was first investigated by Erdős and Rényi \cite{Erdos3}. They proved that the length of the longest run of heads in $n$ tosses of a fair coin will approximately be equal to $\log n$. The question of fluctuation for such stochastic processes was first studied by Chung and Hunt \cite{Chung-Hunt}, and later by Chung and Erdős \cite{Erdos1}, and Cs\'aki, Erdős and R\'ev\'esz \cite{Erdos2}. They found bounds on the length of the longest excursion \footnote{There is a subtle difference between fluctuation and excursion. The former one refers to strict inequality, whereas the later one does not require the inequality to be strict.} in a symmetric random variable.

While our primary focus is on various ergodic averages, we will also study Lebesgue differentiation, 
dyadic martingales, uniform distribution, and weighted averages. It turns out that for these stochastic processes, there are functions $f$ for which $T_Nf$ fails to fluctuate on a set of positive measure of $x$; nonetheless, for the generic function $f$, the stochastic process $T_Nf$ fluctuates around the limit infinitely often a.e.



\section{Organization of the paper}\label{Organization} 

The rest of the paper is organized as follows:

In Section~\ref{Results}, we will state the main results. In Section~\ref{Preliminaries}, we discuss some preliminary concepts that will be needed for the rest of the paper. In Section~\ref{Fluctuation}, we will prove the results related to non-monotonicity and fluctuation of ergodic averages, including \Cref{updown}, \Cref{thm:characterization}, \Cref{thm:generic fluctuation} and \Cref{thm:5}. We deal with the fluctuation of convolution operators in Section~\ref{Derive}. In this section, we prove \Cref{thm:bothwaysae}. \Cref{thm:osc} is proved in Section~\ref{Mart}. In Section~\ref{UD} and Section~\ref{GenCET}, we discuss the non-monotonicity and fluctuation of uniformly distributed sequences and weighted ergodic averages, respectively.

\section{Summary of results}\label{Results}

In this section we state a variety of results about non-monotonicity and fluctuation.  The proofs for these results are given in the upcoming sections as stated in Section~\ref{Organization}.

As a basic model case, take the classical ergodic averages. We have an ergodic, invertible, measure-preserving mapping $T$ on a standard probability space $(X,\mathcal{B}, \mu, T)$. For a function $f\in L^1= L^1(X,\mathcal{B}, \mu)$, 
by the Pointwise Ergodic Theorem, $\setA_{[N]} f(T^n x)$ converges to $\int f$ for a.e. $x$. (See \eqref{notation} for the meaning of the averaging notation).

In the special case, if we consider the irrational rotation $(\setT,{\Sigma}, \lambda, R_\alpha)$ and a continuous function $f$, we can see that for almost every $x$, the convergence of $\setA_{[N]} f(x+n \alpha)$ to $\int f$ is non-monotone. (See \Cref{non-monotone} for the definition of non-monotonicity.) Indeed, compare $\setA_{[N]} f(x+n \alpha)$ with $\setA_{[N+1]} f(x+n \alpha)$, and use the uniform distribution of $(n\alpha)$ mod 1 in $\setT$. See Section~\ref{Fluctuation} for details.

Our first proposition shows that this is, in fact, true for any ergodic transformation $T$ and any integrable function $f$.

\begin{prop}\label{updown} 
Suppose $f\in L^1$ is real-valued and not a constant function.  Then  \ for a.e. $x$,  the averages $\setA_{[N]} f(T^n x)$ are non-monotone.
\end{prop}
We can also give a somewhat more detailed result about the gap that occurs between two consecutive averages.

\begin{prop}\label{prop:updowngap} 
Suppose $f\in L^1$ is real-valued and not a constant function.  Then   \ there exists $d > 0$ such that for all $T$ ergodic, we have for a.e. $x$, 
\begin{align*}
\setA_{[N+1]} f(T^n x)&> \frac dN + \setA_{[N]} f(T^n x) \text{ infinitely often, and }\\
\setA_{[N+1]} f(T^n x) &< -\frac dN + \setA_{[N]} f(T^n x) \text{ infinitely often.}
\end{align*} 
\end{prop}

Different aspects of fluctuation for the classical ergodic averages have already been studied by several authors. For instance, Bishop ~\cite{Bishop} showed that one can give universal estimates for the
probability that there are many fluctuations in the ergodic averages of $L^1$ functions. In fact, he proved an exact analogy of Doob’s upcrossing inequality for martingales, which was later improved by Kalikow and Weiss ~\cite{K-W}. The most relevant result related to our paper is the following theorem ~\cite[Theorem 4]{Hal}.

\begin{thm}\label{thm:Hal}[Halász]
For any $f\in L^1$, $\big(\setA_{[N]}f(T^n x)-\int f\big)$ almost everywhere changes sign infinitely often in the weaker sense that it cannot be eventually positive or negative.
\end{thm}

Halász's theorem suggests that for every $f\in L^1$, $\setA_{[N]}f(T^n x)$ may fluctuate around the mean infinitely often for almost every $x$. In the following example, we see that it may not be the case.
\begin{exmp}
 Take our probability space to be $[0,1)$ with the usual Lebesgue measure $\lambda$.  Take any ergodic map $T$ and any Lebesgue measurable set $F\subset [0,1)$ with $\lambda(F)>0.$ Let $f$ be the coboundary with the transfer function $\setone_F$, that is, $f = \setone_F-\setone_F\circ T$. Then the possible values of $\setone_F(Tx)-\setone_F\circ T^{N+1}(x)$ for  $x\in T^{-1}F$ are $1-1 =0$ and $1-0=1$.  So for these $x$, we get $\setone_F(Tx) - \setone_F\circ T^{N+1}(x)$ changing value infinitely often, but the sign is never strictly negative.  At the same time, the possible values of $\setone_F(Tx)-\setone_F\circ T^{N+1}(x)$ for  $x\not\in T^{-1}F $ are $0-1 =-1$ and $0-0=0$. So, for these $x$, we get $\setone_F(Tx) - \setone_F\circ T^{N+1}(x)$, again changing value infinitely often, but the sign is never strictly positive. So, we do not have the averages $\setA_{[N]}f(T^n x)$ fluctuating around the mean in either case.
 \end{exmp}

\begin{defn}\label{barrier}
Let $F$ be a real-valued, measurable function. We say that $F$ has an \emph{upper barrier} if there is a set $P$ of positive measure on which $F$ is constant
$p$ and such that $F(y) \leq p$ for a.e. $y \in X$. We say that $F$ has a \emph{lower barrier} if there is a set $Q$ of
positive measure on which $F$ is constant $q$ and such that $F(y) \geq q$ for a.e. $y \in X$.
\end{defn}
In the next theorem, we will characterize those functions for which $\setA_{[N]}f(T^n x)$ fails to fluctuate on a set of positive measure.

\begin{thm}\label{thm:characterization} 
Assume that $F$ has either an
upper barrier or a lower barrier, and $f=F-F\circ T$ is an integrable function such that $\int f=0$. Then there exists a set $E$ of positive measure such that $\setA_{[N]}f(T^n x)$ fails to fluctuate on $E$, that is, for every $x\in E$,  \  either $\setA_{[N]}f(T^n x)\geq 0$ eventually or $\setA_{[N]}f(T^n x)\leq 0$ eventually.

Conversely, assume that for some $f\in L^1$ with $\int f=0$ , $\setA_{ [N]}f(T^n x)$ fails to fluctuate on a set of positive measure. Then $f$ can be expressed as $F-F\circ T$ for some measurable function $F$ that has either an upper or a lower barrier.
\end{thm}

For a given increasing sequence $(N_i)_i$ of positive integers, we initiate the study of similar types of questions for the averages $\setA_{[N_i]}f(T^n x)$. More precisely, we ask the following questions.
\begin{quest}\label{ques:1.3} 
For any real-valued function $f\in L^1$, is the convergence of $\left(\setA_{[N_i]}^Tf\right)_i$ non-monotone?   
\end{quest}
Similarly,
\begin{quest}\label{ques:1.4} 
 For almost every $x$, do the averages $\setA_{[N_i]}f(T^n x)$ fluctuate almost everywhere around its mean $\int f$, that is, does $\big(\setA_{[N_i]}f(T^n x)-\int f\big)_i$ change sign infinitely often?
\end{quest}
One would expect that the non-monotonicity and the fluctuation of the averages \linebreak $\setA_{[N]}f(T^n x)$ is not something that persists when one takes a subsequence $\setA_{[N_i]}f(T^n x)$.
 In \Cref{prop:nonmonotonicity example}, we will construct a system $(X,\mathcal{B},\mu, T)$ and a subsequence $(N_i)$ such that for every $x$, $\Big(\setA_{[N_{i+1}]}f(T^n x)-\setA_{[N_i]}f( T^n x)\Big)_i$ will be eventually positive or negative depending only on $x$ (not on $i$), thereby answering \Cref{ques:1.3} negatively. On the other hand, the following theorem shows that for any increasing sequence of positive integers $(N_i)$ and for the generic function $f$, the averages $\setA_{[N_i]}f(T^n x)$ exhibit a fluctuating behavior around the mean infinitely often a.e.

\begin{thm}\label{thm:generic fluctuation} 
 If $T$ is ergodic, and $(N_i)$ is an unbounded sequence, then there is a dense $G_\delta$ set $\mathcal{O}$ in $ L^1$ such that for every $f\in \mathcal{O}$ and for almost every $x \in X$, $\setA_{ [N_i]}f(T^n x)$   fluctuates infinitely often around its mean. That is, for a.e. $x$, we have $\setA_{ [N_i]}f(T^n x)> \int f$ for infinitely many $i$ and $\setA_{ [N_i]}f(T^n x)< \int f$ for infinitely many (different) $i$.
\end{thm}

A salient feature of the proof of Theorem~\ref{thm:generic fluctuation} is that it uses the
\emph{strong sweeping out property} of lacunary sequences.

The next natural question would be to describe those functions $f$ for which the averages $\setA_{ [N_i]}f(T^n x)$ exhibit non-monotonicity or fluctuating behavior (around the mean) for a set of positive measure. Our results in this regard are far from complete.
It turns out that these questions are related to an intrinsic dynamical property of the sequence $(N_i)$, which we will call the \emph{complete recurrence property}. See Subsection \ref{strong} for the definition of the complete recurrence property.

In the special case when $(N_i)$ is a syndetic sequence (i.e., a sequence that has bounded gaps between the consecutive terms), for a `typical' $\alpha$, we can say something definitive about the non-monotonicity of the averages with respect to the rotation by $\alpha$.

\begin{thm}\label{thm:5} 
Let $(N_i)$ be a syndetic sequence. Then there exists a countable subset $E$ of $\setT$ such that the following holds.\\  If $\alpha\in \setT\setminus E$ and $f$ is a non-constant, integrable function, then for a.e. $x$, the averages $\setA_{ [N_i]}f(x+ n\alpha)$ are non-monotone.
\end{thm}
The next topic of our discussion will be the fluctuation of convolution operators.
\begin{defn}
An \emph{approximate identity} is a sequence $(\phi_n)\in L^1(\setR)$ such that for all $f\in L^1(\setR)$, we have $\lim_{n\to \infty}\|\phi_n\star f-f\|_1=0$. A sequence $(\phi_n)$ is \emph{normalized} if $\|\phi_n\|_1=1$ for all $n\geq 1$. A sequence is \emph{eventually local} if the functions all have bounded support and for all $\delta>0$, there exists $N\geq 1$, such that for all $n\geq N$, we have $\phi_n(x)=0$ a.e. on\linebreak $\setR\setminus[-\delta,\delta].$ A \emph{proper} approximate identity is a normalized, eventually local, approximate identity consisting of positive functions.
\end{defn}

Suppose that $(\phi_n)$ is a proper approximate identity on $L^1(\mathbb
R)$. While the convolutions $\phi_n\star f$ converge in $L^1$-norm to
$f$ for all $f \in L^1(\mathbb R)$, these convolutions may or may
not converge almost everywhere.  This is a topic for which there
is a great deal of literature of all types; see \cite{Ros-Convolution} for
some background and theorems about these issues.

In this paper, we are more
concerned with what could be considered the other side of the issue,
the rate of this convergence. Let $CB(\mathbb R)$
denote the continuous, bounded functions on $\mathbb R$ in 
the uniform norm topology. The main result concerning this topic is the following:
\begin{thm}\label{thm:bothwaysae} 
Let $B=L^p(\mathbb R)$ for some fixed $p, 1 \le p \le \infty$, or $B=CB(\mathbb R)$.  Suppose that $(\phi_n)$ is a proper approximate identity and
$(\epsilon_n)$ is a 
sequence of positive numbers with $\lim\limits_{n \to \infty}
\epsilon_n = 0$.
Then there is a dense $G_{\delta}$ subset $\mathcal A$ of $B$ such that
for all $f \in
\mathcal A$, we have both $\limsup\limits_{n \to \infty} \frac {\phi_n
\star f(x)
-f(x)}{\epsilon_n} = \infty$ a.e. and  $\liminf\limits_{n \to
\infty} \frac
{\phi_n \star f(x) - f(x)}{\epsilon_n} = -\infty$ a.e.
\end{thm}

The final topic of this paper is the fluctuation of convergent martingales. The classical dyadic martingale is probably the best example 
to consider because of its connections to Lebesgue differentiation and, through that,
to dynamical systems.

The dyadic martingale is constructed as follows.  Take $\mathcal D_n$ 
to be the dyadic $\sigma$-algebra whose atoms are the intervals $D_{n,j} =
[\frac j{2^n},\frac {j+1}{2^n}], j = 0,\dots,2^n-1$. The finite $\sigma$-algebras 
$\mathcal D_n$ are increasing, with the measure-theoretic 
completion of $\bigcup\limits_{n\in \setN} \mathcal D_n$
being the Lebesgue $\sigma$-algebra in $[0,1]$.   

Now for $G\in L^1[0,1]$, as $n\to \infty$, we have the conditional expectations $E(G|\mathcal D_n)$ converge
to $G$ in $L^1$-norm and a.e. with respect to Lebesgue measure $\lambda$.  We claim
that this convergence is generically fluctuating a.e.  
\medskip
\begin{thm}\label{thm:osc} 
For a dense $G_\delta$ set $\mathcal O$ of
functions in $L^1[0,1]$, if $G\in \mathcal O$, for a.e.   $x\in [0,1]$, $E(G|\mathcal D_n)(x) > G(x)$ for infinitely
many $n\ge 1$, and $E(G|\mathcal D_n)(x) < G(x)$ for infinitely many (other) $n\ge 1$.  
\end{thm}

 \begin{rem}
 It is clear that we have to consider the fluctuation of this strong type on 
only a subset $\mathcal O$ of $L^1[0,1]$.  Indeed, if $G$ were measurable
with respect to $\mathcal D_m$, then $E(G|\mathcal D_n) = G$ for all $n\ge m$.
So for such functions, there is no fluctuation.  If we vary $m$, then such functions are dense
in $L^1[0,1]$ in $L^1$-norm.
    
 \end{rem}

\section{Preliminaries}\label{Preliminaries}
\subsection{Notation} Let $(X,\mathcal{B},\mu,T)$ be a measure-preserving system and $f\in L^1(X)$. For a finite set of $S$ of positive integers and a sequence $(a_n)$ of positive integers, we use the following averaging notation.
\begin{equation}\label{notation}
\setA_{S}f(T^{a_n} x):= \dfrac{1}{|\# S|}\sum_{n\in S}f(T^{a_n} x).
\end{equation} 
Intervals in the averaging notation mean intervals of integers.\\
In particular,
\begin{equation}\label{notation1}
\setA_{[N]}f(T^{n} x):= \dfrac{f(Tx)+f(T^2 x)+\dots+f(T^N x)}{N}.
\end{equation}
When we are concerned with the function space and not their functional value, we will use $\setA_{S}^T$ where $\setA_{S}^T:X\to X$ is defined by $\setA_{S}^Tf(x)=\setA_{S}f(T^n x).$

$\setT$ stands for the one-dimensional torus, that is, $\setR/\setZ$, and $\lambda$ is reserved for denoting the Lebesgue measure.

\vspace{.5cm}

Sometimes we assume that $T$ is totally ergodic to avoid the following obstacle regarding fluctuation. Suppose that $T$ is not totally ergodic.  Then there is a non-zero $f$ such that $f\circ T = \zeta f$ with $\zeta $ a root of unity.  Say $\zeta^N =1$.  Then $\sum\limits_{n\in [N]} \zeta^n =  0$.  So the same thing occurs with any multiple of $N$, and hence, for example, $\setA_{[N]}f(T^n x) = \setA_{[2N]}f(T^n x) = 0$ for almost every $x$. On the other hand, if we assume total ergodicity, then two averages cannot be equal.
\begin{prop} \label{tot}   
If $T$ is ergodic and for (almost) every $x\in X$, $\setA_{[M]}f(T^n x) =$  $ \setA_{[N]}f(T^n x)$ for some $M > N$ and a non-zero $f\in L^1$, then $T$ is not totally ergodic.
\end{prop}
\begin{proof}  Since $\setA_{[M]}^Tf = \setA_{[N]}^Tf$, we have \[f\circ T^M = \frac MN \sum\limits_{[N]} f\circ T^n - \sum\limits_{[M-1]} f\circ T^n.\]
Let $\mathcal S$ be the linear span of $\{f\circ T,\dots,f\circ T^{M-1}\}$.  We have $f\circ T^M \in \mathcal S$.  That is, $\mathcal S \circ T\subset \mathcal S$.  So, because $\mathcal S$ is finite-dimensional, $\mathcal S$ is invariant under $T$. Since $T$ is ergodic, it follows that there is a non-zero $F\in \mathcal S$ and $\zeta \in \mathbb T, \zeta \not=1,$ such that $F\circ T = \zeta F$.  We claim that in fact $\zeta$ is a root of unity and so $T$ is not totally ergodic.  To see this, we take $\displaystyle F = \sum\limits_{m\in[M-1]} c_m f\circ T^m$ for some $c_1,\dots, c_{M-1} \in \mathbb C$.  So we have \begin{align*}
\setA_{[M]}^TF = \sum\limits_{m\in[M-1]} c_m &\setA_{[M]}^T (f \circ T^m)=\sum\limits_{m\in[M-1]} c_m \setA_{[M]}^T f \circ T^m \\ &= \sum\limits_{m\in[M-1]} c_m \setA_{[N]}^T f \circ T^m=  \sum\limits_{m\in[M-1]} c_m \setA_{[N]}^T (f \circ T^m)= \setA_{[N]}^TF.
\end{align*}
Because $F$ is not zero and $F\circ T = \zeta F$, we have $\displaystyle \frac 1M \sum\limits_{m\in[M]} \zeta^m = \frac 1N\sum\limits_{m\in[N]} \zeta^m$.  Using the geometric sum  formula, it follows that $\displaystyle \frac 1M (\zeta^M - 1) = \frac 1N (\zeta^N -1)$.  That is, $M-N = M\zeta^N - N\zeta^M = (M - N \zeta^{M-N})\zeta^N$.  Since $|\zeta| = 1$, it follows that $M -N= |M - N \zeta^{M-N}|$.  That is, in $\mathbb C$ the distance from $N$ to $M$ is the distance of $N\zeta^{M-N}$ to $M$.  But this can only happen if $\zeta^{M-N} =1$.  That is, $\zeta$ is a root of unity.
\end{proof}

\begin{cor}\label{cor:nonconstant} 
If $T$ is totally ergodic, and $f$ is a non-constant $L^1$ function, then $\setA_{[M]}^Tf$ is also non-constant for all $M\geq 1$.
\end{cor}
\begin{proof}
Suppose, to the contrary, there is a constant $c$ such that $\setA_{[M]}^Tf(x)=c$ for almost every $x\in X$. Applying this equality on $T^M x$, we get that $\setA_{[M]}^Tf(T^M x)=c$ for almost every $x$. Adding these two equalities and dividing both sides by $2$, we have $\setA_{[2M]}^T f(x)=c$ for almost every $x$. Thus we have $\setA_{[M]}^Tf(x)=\setA_{[2M]}^T f(x)=c$ for almost every $x$. But this is a contradiction to Proposition \ref{tot}.
\end{proof}

\subsection{Complete recurrence property}\label{strong}
In this subsection, we will introduce the notion of complete recurrence property. This has some consequences related to the non-monotonicity and fluctuation of the subsequence ergodic averages (see \Cref{prop:subsequencefluctuation} and \Cref{thm:6}). It is also a topic of independent interest.

\begin{defn}[Complete recurrence property]
A sequence $(N_i)$ of positive integers is said to satisfy the complete recurrence property for the system $(X,\mathcal{B},\mu, T)$ if every set $E$ of positive measure and every positive integer $L$ satisfy the following condition:
\begin{equation}\label{eq:recurrence}
\displaystyle \mu\Big(\cup_{i\geq L}T^{-N_i}E\Big)=1.
\end{equation}
In particular, if a sequence $(N_i)$ satisfies \eqref{eq:recurrence} for the irrational rotation $(\setT,\Sigma,\lambda,R_\alpha)$, then we say that $(N_i)$ satisfies the complete recurrence property for $\alpha$. We say that a sequence $(N_i)$ satisfies the complete recurrence property if it satisfies the same for \emph{every} ergodic system.
\end{defn}
An equivalent formulation of the complete recurrence property is the following:\\
For every positive integer $L$ and every set $E$ and $F$ of positive measures, there exists $i\geq L$ such that
\begin{equation}\label{eq:equivalent}
\mu\big(E\cap T^{-N_i}F\big)>0.
\end{equation}
 \begin{rem}{}{}
\begin{enumerate}
    \item Recall that a set $R\subset \setN$ is said to be a \emph{set of recurrence} if for any measure-preserving system $(X,\mathcal{B}, \mu, T)$ and for all $A\in \mathcal{B}$ with $\mu(A)>0$ if we have $\mu(A\cap T^{r}A)>0$ for some $r\in R$. This suggests to us the name complete recurrence property. {Sets of recurrence} have been studied extensively in ergodic theory. For a beautiful account on the set of recurrence, the reader is suggested to look at the book of Furstenberg ~\cite{Fur}.
\item While the set of recurrence is defined for any measure-preserving system, to define the complete recurrence property, we must restrict ourselves to ergodic systems.
\end{enumerate}
 \end{rem}
We will give some examples of complete recurrence sequences via the following lemma.
\begin{lem}\label{lem:completerecurrence}
Let $(X,\mathcal{B},\mu,T)$ be an ergodic system and $(N_i)$ be an increasing sequence of positive numbers. Suppose that for every $f\in L^\infty$
\begin{equation}\label{eq:complete}
 \frac{1}{K}\sum_{i\in[K]}f(T^{N_i} x) \to \int f \text{ in $L^2$ norm as } K\to \infty.
\end{equation}

Then $(N_i)$ satisfies the complete recurrence property for $(X,\mathcal{B},\mu,T)$.
\end{lem}
\begin{proof}
To the contrary, suppose that $(N_i)$ does not satisfy the complete recurrence property for $(X,\mathcal{B},\mu,T)$. Then, by \eqref{eq:equivalent}, there exist $E,F \in \mathcal{B}$ with $\mu(E)\cdot \mu(F)>0$ such that
\begin{equation}\label{eq:2}
\mu\big(E\cap T^{-N_i}F\big)=0 \text{ for every large }i.
\end{equation}
So, there exists a set $\tilde{E}\subset E$ with $\mu(E)=\mu(\tilde{E})>0$ and $\tilde{E}\cap T^{-N_i}F= \emptyset $ for every large $i$.
This implies that for every $x\in \tilde{E}$, $\displaystyle \frac{1}{K}\sum_{i\in[K]}\setone_F (T^{N_i} x)\to 0$. On the other hand, using \eqref{eq:complete}, we can find a subsequence $(K_l)$ such that for almost every $x\in X$, $\displaystyle\frac{1}{K_l}\sum_{i\in[K_l]}\setone_F(T^{N_i} x) \to \mu(F)$ as $l\to \infty$. This is a contradiction because $\mu(\tilde{E}) \cdot \mu(F)>0.$
\end{proof}

A particular class of sequences that satisfies the above hypothesis is the class of \emph{ergodic sequences}.
\begin{defn}(Ergodic sequence)
Let $T$ be an ergodic transformation on the probability space $(X,\mathcal{B},\mu)$. A sequence $(N_i)$ of positive increasing integers is said to be an ergodic sequence for the system $(X,\mathcal{B},\mu,T)$ if for every $f\in L^2$, $$\displaystyle \frac{1}{K}\sum_{i\in[K]}f(T^{N_i} x) \to \int f \text{ in $L^2$ norm as } K\to \infty.$$ If the sequence $(N_i)$ is ergodic for every ergodic system, then it is said to be an ergodic sequence.
\end{defn}
For a detailed discussion on ergodic sequences, one can look at the survey article of Rosenblatt and Wierdl ~\cite[Chapter II]{RW}. Here we will list some examples of ergodic sequences, and hence sequences of complete recurrence by \Cref{lem:completerecurrence}.

\begin{exmp}
 Let  $c>1$ be a non-integer real number and $p_n$ denote the $n$-th prime number.
\begin{enumerate}[(a)]
    \item Then from \cite[ Theorem A]{BKQW}, it follows that $(\lfloor n^c \rfloor)_{n\in \setN}$ is an ergodic sequence.
    \item From \cite[Theorem 3.1]{BKGM}, it follows that $(\lfloor p_n^c \rfloor)_{n\in \setN}$ is an ergodic sequence. 
\end{enumerate}
\end{exmp}
\begin{exmp}
Let $p_i$ denote the $i$-th prime number. It is well known that both $(i^2)$ and $(p_i)$ are ergodic sequences for totally ergodic systems.   
\end{exmp}
The following example shows that the complete recurrence property is indeed a special property.

\begin{exmp}
Let $I=[a,b]\subset \setT$ be an interval with $\lambda(I)<\frac{1}{4}$. Let $\alpha$ be a fixed irrational number. Define $S=\{n\in \setN: I\cap (I-n\alpha)\neq \emptyset\}$. We claim that $S$ is a set of recurrence, but it does not satisfy the complete recurrence property. Let $(X,\mathcal{B},\mu,T)$ be any measure-preserving system. Let $(\Tilde{X},\Tilde{\mathcal{B}},\tilde{\mu},\Tilde{T})$ be the product system of $(X,\mathcal{B}, \mu, T)$ and $(\setT,\beta, \lambda,R_\alpha ).$ For an arbitrary $A\in \mathcal{B}$, let us consider the set $\Tilde{A}=A\times I$. By the Poincaré recurrence theorem, there exists $n\in \setN$ such that $\tilde{\mu}\Big( \Tilde{A}\cap \tilde{T}^{-n}\Tilde{A}\Big)>0.$ Such $n$ must belong to $S$, and it also satisfies that $\mu(A\cap T^{-n} A)>0$. Thus $S$ is a set of recurrence. On the other hand, there exists an interval $J$ with $\lambda(J)\leq \frac{1}{4}$ such that $\{n\alpha: n\in S\}\subset J$. Hence, for any interval $E$ with $\lambda (E)<\frac{3}{4}$, we must have $\lambda(J+E)<1.$ So, $\lambda\Big(\cup_{n\in S}(E-n\alpha)\Big)<1$. Hence, $S$ cannot satisfy the complete recurrence property. \qed
\end{exmp}

\section{Fluctuation and non-monotonicity of the classical ergodic averages}\label{Fluctuation}

\subsection{Proof of the fluctuation and non-monotonicity results}

We begin by proving \Cref{updown}.

First observe that $\setA_{[N+1]}f(T^n x) \le \setA_{[N]}f(T^n x)$ implies that
\[\frac 1{N+1} f(T^{N+1}x) \le (\frac 1N - \frac 1{N+1}) \sum\limits_{[N]} f(T^nx).\]  This means that $f(T^{N+1}x) \le \setA_{[N]}f(T^n x).$   Indeed, $\setA_{[N+1]}f(T^n x) \le \setA_{[N]}f(T^n x)$ holds if and only if $f(T^{N+1}x) \le \setA_{[N]}f(T^n x)$ holds.  Similarly, $\setA_{[N+1]}f(T^n x) \ge \setA_{[N]}f(T^n x)$ holds if and only if we have $f(T^{N+1}x) \ge \setA_{[N]}f(T^n x)$.
Now if either of these inequalities hold for all large enough values of $N$, then we can compare $f(T^{N+1}x)$ to $\setA_{[N]}f(T^n x)$ to arrive at a contradiction.  

\begin{proof}[Proof of \Cref{updown}]  Without loss of generality, we can assume that $f$ is mean-zero but not constant.  Hence, there is some $\delta > 0$, and there are sets $A$ and $B$ of positive measure, such that $f\ge \delta $ on $A$ and $f\le -\delta$ on $B$.  By ergodicity, for a.e. $x$, there is some $N_0$, depending on $x$, such that $|\setA_{[N]}f(T^n x)| \le \delta/2$ for all $N\ge N_0$.  In addition, the ergodicity of $T$ implies that for a.e. $x$, we have $T^{N+1}x \in A$ for infinitely many $N$, and $T^{N+1}x \in B$ for infinitely many $N$.  Therefore, for a.e. $x$,  we cannot have for all $N\ge N_0$,  $f(T^{N+1}x) \le \setA_{[N]}f(T^n x)$ because using recurrence into $A$ this would imply $\delta \le \delta/2$.  But at the same time, for a.e. $x$, we cannot have for all $N\ge N_0$, $f(T^{N+1}x) \ge \setA_{[N]}f(T^n x)$ because using recurrence into $B$ this would imply $-\delta \ge -\delta/2$. Using both of these computations, we see that for a.e. $x$, for infinitely many $N$, we do not have $\setA_{[N+1]}f(T^n x) \le \setA_{[N]}f(T^n x)$, and for infinitely many $N$, we do not have $\setA_{[N+1]}f(T^n x) \ge \setA_{[N]}f(T^n x)$.  Hence, for a.e. $x$, the ergodic averages are not monotonically decreasing, and they are not monotonically increasing.
\end{proof}
Now, we will prove \Cref{prop:updowngap}.
\begin{proof}[Proof of \Cref{prop:updowngap}]
Without loss of generality, we can assume that $f$ is mean-zero but not constant.  Hence, there is some $\delta > 0$, and there are sets $A$ and $B$ of positive measure, such that $f\ge \delta $ on $A$ and $f\le -\delta$ on $B$. Let $d= \frac{\delta}{6}$. If we have $\displaystyle \setA_{[N+1]}f(T^n x)\le \frac dN + \setA_{[N]}f(T^n x)$ for all sufficiently large $N$, then using the computation above, we have $\displaystyle f(T^{N+1}x) \le \frac {(N+1)d}N + \setA_{[N]}f(T^n x)$ for all sufficiently large $N$.  But for a.e. $x$, $T^{N+1}x \in A$ for infinitely many $N$.  Also, for  a.e. $x$,  $\setA_{[N]}f(T^n x) \to \int f\, d\mu = 0$ as $N\to \infty$.  Hence, for a.e. $x$, we have for infinitely many $N$,  
$\displaystyle \delta \le 2d + \setA_{[N]}f(T^n x) \le 2d + \frac {\delta}2 = \frac {\delta}3 + \frac {\delta}2$.  This is impossible.  So we must have for a.e. $x$, $\setA_{[N+1]}f(T^n x) > \frac dN + \setA_{[N]}f(T^n x)$ for infinitely many $N$.

Similarly, if we have $\setA_{[N+1]}f(T^n x)\ge -\frac dN + \setA_{[N]}f(T^n x)$ for all sufficiently large $N$, then using the computation above, we have $f(T^{N+1}x) \ge -\frac {(N+1)d}N + \setA_{[N]}f(T^n x)$ for all sufficiently large $N$.  But for a.e. $x$, $T^{N+1}x \in B$ for infinitely many $N$.  Also, for a.e. $\setA_{[N]}f(T^n x) \to \int f\, d\mu = 0$ as $N\to \infty$. Hence,  for a.e. $x$, we have for infinitely many $N$, 
$-\delta \ge -2d + \setA_{[N]}f(T^n x) \ge -2d - \frac {\delta}2 = -\frac {\delta}3 - \frac {\delta}2$.  This is impossible.  So we must have for a.e. $x$,  $\setA_{[N+1]}f(T^n x) < -\frac dN + \setA_{[N]}f(T^n x)$ for infinitely many $N$.
\end{proof}

 We need the following results to prove \Cref{thm:characterization}. See \ref{barrier} for the definitions of upper barrier and lower barrier.

\begin{prop}\label{prop:subsequencefluctuation} 
Suppose that $(N_i)$ satisfies the complete recurrence property for the system $(X,\mathcal{B},\mu,T)$. Let $F$ be a measurable function such that $F$ has no
upper barrier and no lower barrier. Also, assume that $f = F-F\circ T$ is an integrable function with $\int f=0$. Then for a.e. $x$, the averages $\setA_{ [N_i]}f(T^n x) < 0$ infinitely often and $\setA_{ [N_i]}f(T^n x) > 0$ infinitely often.
\end{prop}

\begin{proof}

Since $F$ does not have an upper barrier, $F\circ T$ also does not have an upper barrier. Hence, for a.e. $x$, there exists a set $B=B(x)$ of positive measure such that $F(Ty)>F(Tx)$ for all $y\in B$. Since $(N_i)$ satisfies the complete recurrence property, for every positive integer $L$, we have
$\displaystyle \mu\Big(\cup_{i\geq L}T^{-N_i}B\Big)=1.$ Thus, for infinitely many $i$, $T^{N_i}x$ falls into $B$. For such an $i$, $\setA_{ [N_i]}f(T^{n}x)=\frac{1}{N_i} \Big(F(Tx)-F(T^{N_i+1}x)\Big)<0$. Similarly,  for infinitely many $j$, $\setA_{ [N_j]}f(T^{n}x)>0$.
\end{proof}

Since $\setN$ (with respect to the natural order) satisfies the complete recurrence property, we have the following corollary.
\begin{cor}\label{cor:1} 
 Let $F$ be a measurable function such that $F$ has no
upper barrier and no lower barrier. Also, assume that $f = F-F\circ T$ is an integrable function with $\int f=0$. Then for a.e. $x$, the averages $\setA_{ [N]}f(T^n x) < 0$ infinitely often and $\setA_{ [N]}f(T^n x) > 0$ infinitely often.
\end{cor}

 \begin{rem}
One cannot expect that the conclusion of \Cref{cor:1} will hold for \emph{every} $x\in X$. In fact, in contrast to \Cref{cor:1}, it was proved by Peres ~\cite{Peres} that if $X$ is a compact space, $\mu$ is a Borel probability measure on $X$, $T:X\rightarrow X$ a measure-preserving continuous transformation, and $g:X\rightarrow \setR$ a continuous function, then for some $y\in X$ 
\begin{equation}\label{eq:heavy}
\frac{1}{N}\sum_{n\in [N] }g (T^n y)\geq \int g d\mu \text{ for all }N\geq 1.
\end{equation}
The collection of those $y$ which satisfy \eqref{eq:heavy} is called the \emph{Heavy set} for $g$. See also ~\cite{BR} for further results on the Heavy set.
 \end{rem}
Now, we are ready to prove  Theorem~\ref{thm:characterization}. This proof is based on the argument of \cite[Theorem 4]{Hal}.
\begin{proof}[Proof of Theorem~\ref{thm:characterization}]
Assume that $F$ has an
upper barrier (the lower barrier case is similar). Then by definition, there is a set $E$ of positive measure on which $F$ is constant
$p$ and $F(x) \leq p$ for a.e. $x \in X$. Now consider the function $f=F-F\circ T$. Observe that  $\displaystyle\setA_{ [N]} f (T^n x)\geq 0$ for all $x\in E$.

Conversely, assume that there exists a set $E\in \mathcal{B}$ such that $\mu(E)>0$ and for all $x\in E$, $\setA_{ [N]}f(T^n x) \geq 0$ eventually (the $`\leq'$-case is similar). This implies that $\setA_{ [0,N)}f(T^n x) \geq 0$ eventually on $T^{-1 }E.$ Our goal is to show that $f= F- F\circ T$ for some measurable function $F$ that also has an upper or lower barrier. Consider the following function:
\begin{equation}
    F(x)=\inf_{N\in \setN} \sum_{n\in [0,N)}f(T^n x).
\end{equation}
Observe that $F$ is bounded below on $T^{-1 }E$; otherwise we cannot have $\setA_{ [0,N)}f(T^n x) \geq 0$ eventually on $T^{-1 }E$. Since the set $G:=\{x: F(x) \text{ is bounded below}\}$ is invariant under $T$, and $T$ is ergodic, $\mu(G)=1$. So, $F$ is a measurable function that is finite a.e. We have
\begin{equation}\label{eq:Hal:1}
\begin{aligned}
F(Tx)= \inf_{N\in \setN} \sum_{n\in [N]}f(T^n x)&=\Big(\inf_{N\in \setN} \sum_{n\in [0,N]}f(T^n x)\Big)-f(x)\\
&\geq \Big(\inf_{N\in \setN} \sum_{n\in [0,N)}f(T^n x)\Big)-f(x)\\
&= F(x)- f(x).
\end{aligned}
\end{equation}
Define $K(x)= f(x)+ F(Tx)- F(x)$.
By \eqref{eq:Hal:1}, we have $K(x)\geq 0$ for almost every $x$, and by the pointwise ergodic theorem
\begin{equation}
    \dfrac{1}{N}\sum_{n\in [0,N)}K(T^n x) \to \int_X K(x) d\mu \text{ for a.e. }x.
\end{equation}
(even if $\int_X K(x) d\mu= \infty$, it follows from the finite case).
The left-hand side is 
\begin{equation}
\dfrac{1}{N}\sum_{n\in [0,N)}f(T^n x)+ \dfrac{1}{N}\sum_{n\in [0,N)}\Big(F(T^N x)- F(x) \Big).
\end{equation}
The first term tends to $\int f=0$ a.e. by the pointwise ergodic theorem. Since $F(x)$ is finite a.e., there exists $M>0$ such that the set $S(M)=\{x: |F(x)|\leq M\}$ is of positive measure, and hence for almost every $x$, $T^N x$ falls into $S(M)$ for infinitely many $N$. For such a $N$, $\big(F(T^N x)- F(x)\big) $ is bounded by $2M$, and so for almost every $x\in S(M)$ the second term also tends to $0$. Letting $M\to \infty$, we conclude that $\int_X K(x) d\mu= 0$. Since $K(x)$ is non-negative, we must have $K(x)=0 \text{ a.e.}$
This proves that $f= F- F\circ T$ almost everywhere. By assumption, we have $\setA_{ [N]}f(T^n x) \geq 0$ eventually on a set of positive measure $E$. Invoke \Cref{cor:1} to conclude that $F$ must have either an upper barrier or a lower barrier, finishing the proof of the theorem.
\end{proof}

We need the following lemma of Halász ~\cite[Theorem 2]{Hal} for proving the next corollary.
\begin{lem}\label{thm:4} 
Let $(X,\mathcal{B},\mu, T)$ be a totally ergodic system. If for a measurable set $E$, $\displaystyle\Big(\sum_{n\in [N]}\setone_E(T^n x)-N\mu(E)\Big)_{N\in \setN}$  is bounded below on a set of positive measure, then $e^{2\pi i \mu(E)}$ is an eigenvalue of $T$, that is , there exists a function $g(x)$ not identically $0$ a.e. such that $g(Tx)= e^{2\pi i \mu(E)} g(x)$.
\end{lem}

\begin{cor} 
Let $(X,\mathcal{B}, \mu, T)$ be a measure preserving system such that $T$ is totally ergodic. Then for every $E\in \mathcal{B}$ with $0<\mu(E)<1$, $\displaystyle\setA_{ [N]}\setone_E(T^n x)$ fluctuates infinitely often around $\mu(E)$ a.e. 
\end{cor}

\begin{proof}
Suppose, to the contrary, $\displaystyle\setA_{ [N]}\setone_E(T^n x)$ fails to fluctuate on a set of positive measure $F$. Then we claim that either,
\begin{equation}\label{eq:new1}
\Big(\sum_{n\in [N]}\setone_E(T^n x)-N\mu(E)\Big)_{N\in \setN} \text{ is bounded below on a set of positive measure}
\end{equation}
or,
\begin{equation}\label{eq:new2}
\Big(\sum_{n\in [N]}\setone_{E^c}(T^n x)-N\mu(E^c)\Big)_{N\in \setN} \text{ is bounded below on a set of positive measure.}
\end{equation}
 Indeed, if for almost every $x$, we have
\begin{equation}\label{new:1}
\inf_{N\in \setN}\Big(\displaystyle\sum_{n\in [N]}\setone_E(T^n x)-N\mu(E)\Big)=-\infty,
\end{equation} 
then for almost every $x$, we have $\displaystyle\setA_{ [N]}\setone_E(T^n x)-\mu(E)<0$ infinitely often. Similarly, if \eqref{eq:new2} does not hold, then for almost every $x$, $\displaystyle\setA_{ [N]}\setone_E(T^n x)-\mu(E)>0$ infinitely often. By assumption both cannot occur simultaneously.

By Lemma~\ref{thm:4}, either $e^{2\pi i \mu(E)}$ or $e^{2\pi i \mu(E^c)}$ is an eigenvalue of $T$. Again, since $T$ is ergodic, by \Cref{thm:Hal} for almost every $x$ we can find two subsequences $N_i=N_i(x)$ and  $N_j=N_j(x)$ such that $\setA_{ [N_i]}\setone_E (T^n x)\geq \mu(E)$ and $\setA_{ [N_j]}\setone_E(T^n x)\leq \mu(E)$. Hence, the failure of fluctuation imply that for all $x\in F$, we can find a subsequence $N_{i}=N_i(x)$ such that $\setA_{ [N_i]}\setone_E(T^n x)=\mu(E).$ Since $\setA_{ [N]}\setone_E(T^n x)$ is always a rational number, the equality can hold only if $\mu(E)$ (and $\mu(E^c)$) are rational numbers.  This contradicts the total ergodicity of $T$ and thereby finishes the proof of the corollary. 
\end{proof}

The questions of non-monotonicity and fluctuation for the subsequence of ergodic averages are much more subtle than that of the classical ergodic averages. As the following example shows, one can construct a totally ergodic system $(X,\mathcal{B},\mu,T)$ (which is going to be an irrational rotation on the usual one-dimensional torus for us), a function $f\in L^1(X)$, and a sequence $(N_i)$ of positive integers such that for almost every $x$, the averages $\displaystyle\setA_{ [N_i]}f(T^n x)$ are eventually, strictly monotone.
\begin{prop}\label{prop:nonmonotonicity example} 
  Let $\alpha$ be an irrational number and let the sequence $(N_i)$ be
  strictly increasing and satisfy $0<(N_{i}+1)\alpha<(N_{i-1}+1)\alpha\mod 1$
  for every $i$.

  Then there is a function $f\in L^1(\mathbb T,\lambda)$ 
  such that for almost every $x\in \setT$ we have
  \begin{equation}
    \label{eq:non-mon}
    \setA_{[N_{i-1}]}f(x+n\alpha)-\setA_{[N_i]}f(x+n\alpha)>0
    \text{ for }i\ge i(x). 
  \end{equation}

  In particular, the sequence
  $\big(\setA_{[N_i]}f(x+n\alpha)\big)_i$ of averages  is almost everywhere
  strictly monotone.
\end{prop}
\begin{proof}
  In fact, we identify a certain set $F_+$ of coboundaries the
  elements of which will satisfy the claim of the proposition, that
  is, \eqref{eq:non-mon} for $\lambda$-a.e. $x$.
  
  Let us denote by $G_+$ those functions $g\in L^1$ which are
  piecewise increasing, that is, there is a partition of $\mathbb T=[0,1)$
  into finitely many intervals\footnote{We assume that these intervals
    are subintervals of the interval $[0,1)$.} so that on each
  interval $g$ is strictly increasing. Let $F_+$ be the set of
  coboundaries with transfer functions from $G_+$,
  \begin{equation*}
    F_+=\{f: f(x)=g(x+\alpha)-g(x), g\in G_+\}.
  \end{equation*}
  Let
  $f\in F_+$ and let $g\in G_+$ be the corresponding transfer
  function, so $f(x)=g(x+\alpha)-g(x)$. Let $x$ be in the interior of
  one of the intervals $I$ where $g$ is strictly increasing.  Then
  there is an index $i=i(x)$ so that $x+(N_{i-1}+1)\alpha\in I$ for all
  $i\ge i(x)$.  By definition, we have
  \begin{align*}
    \setA_{[N_{i-1}]}f(x+n\alpha)-&\setA_{[N_i]}f(x+n\alpha)\\
    &=\frac1{N_{i-1}}\Big(g(x+(N_{i-1}+1)\alpha)-g(x)\Big) - \frac1{N_{i}}\Big(g(x+(N_{i}+1)\alpha)-g(x)\Big)
  \end{align*}
  If $i\ge i(x)$ then $g\big(x+(N_{i-1}+1)\alpha\big)>g\big(x+(N_{i}+1)\alpha\big)$ since we
  assumed that $(N_{i-1}+1)\alpha>(N_{i}+1)\alpha$ for every $i$, hence
  $x+(N_{i-1}+1)\alpha>x+(N_{i}+1)\alpha$ for $i\ge i(x)$, and $g$ is strictly
  increasing on the interval $I\subset [0,1)$.  But then
  \[g\big(x+(N_{i-1}+1\big)\alpha)-g(x)>g\big(x+(N_{i}+1\big)\alpha)-g(x)\] as well. Since
  $\frac1{N_{i-1}}>\frac1{N_i}$, we have
  \begin{equation*}
    \frac1{N_{i-1}}\Big(g\big(x+(N_{i-1}+1)\alpha\big)-g(x)\Big)>
    \frac1{N_{i}}\Big(g\big(x+(N_{i}+1)\alpha\big)-g(x)\Big)\text{ for }i\ge i(x)
  \end{equation*}
  which then implies
  \begin{equation*}
    \setA_{[N_{i-1}]}f(x+n\alpha)-\setA_{[N_i]}f(x+n\alpha)>0
    \text{ for }i\ge i(x).
  \end{equation*}
\end{proof}
 \begin{rem}
$G_+$ is fairly large, namely it is dense in $L^1$, since every
    continuous function can be approximated in $L^1$-norm by functions
    from $G_+$.  It follows that $F_+$ is dense in the orthocomplement
    of the constant functions, that is, in the set of zero mean $L^1$
    functions.   
 \end{rem}

\medskip

Next, we will prove \Cref{thm:generic fluctuation}. We need a couple of basic facts about the \emph{strong sweeping out property}.

\begin{defn}[Strong sweeping out property]
Let $(X,\mathcal{B},\mu, T)$ be a measure-preserving system. Let $(M_n)$ be a sequence of linear operators from $L^1$ to $L^1$. We say that the \emph{strong sweeping out} property holds for the operators $M_n$ if for every $\epsilon>0$ there is a set $E\in \mathcal{B}$ with $\mu(E)<\epsilon$ and 
\begin{equation*}
\limsup_{N \to \infty} M_N\setone_E(x)=1 \text{ a.e. and } \liminf_{N \to \infty} M_N \setone_E(x)=0 \text{ a.e.}
\end{equation*}

A sequence $(a_n)$ is said to satisfy the \emph{strong sweeping property} if for every aperiodic system $(X,\mathcal{B},\mu, T)$ and for every $\epsilon>0$, there exists a set $E$ with $\mu(E)<\epsilon$ such that
\begin{equation*}
\limsup_{N\to \infty}\setA_{[N]}\setone_E(T^{a_n}x)=1 \text{ a.e. and } \liminf_{N\to \infty} \setA_{[N]}\setone_E(T^{a_n}x)=0 \text{ a.e.}
\end{equation*}
\end{defn}
\begin{lem}\label{lem:sso}
Suppose that $(a_n)$ is a lacunary sequence; that means, there exists $\rho>1$ such that $\frac{a_{n+1}}{a_n}\geq \rho$ for all large $n$. Then $(a_n)$ satisfies the strong sweeping out property.
\end{lem}
\begin{proof}
This is proved in Corollary 1.2 of ~\cite{ABJLRW}. See also \cite{JO} or \cite{MRW} for a more transparent proof.
\end{proof}

\begin{proof}[Proof of \Cref{thm:generic fluctuation}]
Define
\begin{align*}
S_M^+=\Big\{f\in L^1: \mu\{x:\setA_{ [N_i]}f(T^n x)>0 \text{ for some } i\geq M\}>1-\frac{1}{M}\Big\} \intertext{ and }
S_M^-=\Big\{f\in L^1: \mu
\{x:\setA_{[N_i]}f(T^n x)<0 \text{ for some } i\geq M\}>1-\frac{1}{M}\Big\}.
\end{align*}
We will prove that $S_M^+$ and $S_M^-$ are both open and dense sets.
 If we can show this, then $\displaystyle{\mathcal{O}}=\cap_{M}S_M^+\cap {S}_M{^-}$ will satisfy the desired properties of the theorem. Notice that $f\in S_M^{+}$ iff $-f\in S_M^{-}$. Hence, we will work only with $S_M^{+}$. The proof of the theorem follows from \Cref{prop:osc1} and \Cref{prop:osc2}.

\begin{prop}\label{prop:osc1}
$S_M^{+}$ is open in $L^1$.
\end{prop}
\begin{proof}
Let $f\in S_M^{+}$. Choose $\rho>0$ satisfying $\mu(D)>\rho>1-\frac{1}{M}$, where
\begin{equation*}
D=\{x:\setA_{ [N_i]}f(T^n x)>0 \text{ for some } i\geq M\}.
\end{equation*}
Define 
\begin{equation*}
D_L:= \{x:\setA_{ [N_i]}f(T^n x)>0 \text{ for some } i \text{ with }L\geq i\geq M\}.
\end{equation*}
Observe that $D_L \uparrow D$. Hence, there exists $L_0$ such that $\mu(D_{L_0})>\rho.$
Similarly, define
\begin{equation*}
D_{L_0}^{\delta}:= \{x:\setA_{ [N_i]}f(T^n x)>\delta \text{ for some } L_0\geq i\geq M\}.
\end{equation*}
Observe that as $\delta \downarrow 0$, $D_{L_0}^{\delta} \uparrow D_{L_0}.$ Hence, there exists $\delta_0>0$ such that $\mu(D_{L_0}^{\delta_0})>\rho$.
Choose $$\epsilon= \frac{1}{4}\Big(\rho-\big(1-\frac{1}{M}\big)\Big)\frac{\delta_0}{(M-{L_0})}.$$ We claim that whenever $\|f-g\|_1<\epsilon$, $g\in S_M^{+}$. Let $h= g-f$. By Markov's inequality, for every $i\in \setN$ we have 
\begin{equation*}
\mu\big\{x: \setA_{ [N_i]}h(T^n x)<-\frac{\delta_0}{2}\big\}<\frac{2\|h\|_1}{\delta_0}.
\end{equation*}
Hence, 
\begin{equation*}
\mu \Big\{x: \setA_{ [N_i]}h(T^n x)<-\frac{\delta_0}{2} \text{ for some }i \text{ with }L_0\geq i \geq M \Big\}< \frac{2\|h\|_1 (M-{L_0})}{\delta_0}
\end{equation*}
Since $\|h\|_1<\epsilon$, we must have
\begin{equation*}
\mu \Big\{x: \setA_{ [N_i]}h(T^n x)<-\frac{\delta_0}{2} \text{ for some }i \text{ with }L_0\geq i \geq M \Big\}< \frac{1}{2}\Big(\rho-\big(1-\frac{1}{M}\big)\Big)
\end{equation*}
This implies that
\begin{equation*}
\mu \Big\{x: \setA_{ [N_i]}g(T^n x)>\frac{\delta_0}{2} \text{ for some }L_0\geq i \geq M \Big\}>\rho- \frac{1}{2}\big(\rho-(1-\frac{1}{M})\big).
\end{equation*}
By our choice of $\rho$, $\rho- \frac{1}{2}\Big(\rho-\big(1-\frac{1}{M}\big)\Big)>\big(1-\frac{1}{M}\big)$. Thus, we have \begin{equation*}
\mu \Big\{x: \setA_{ [N_i]}g(T^n x)>0 \text{ for some } i \geq M \Big\}> \Big(\rho-\big(1-\frac{1}{M}\big)\Big).
\end{equation*}
This completes the proof of the proposition.
\end{proof}
\begin{prop}\label{prop:osc2}
$S_M^+$ is dense in $L^1$.
\end{prop}
\begin{proof}
 It will be sufficient to show that $S_M^+$ is dense in $L^1_0$, where $L^1_0$ denotes the space of $L^1$ functions with mean $0$. Since $T$ is ergodic, coboundaries are dense in $L^1_0$. Furthermore, $L^\infty$ is dense in $L^1$. Hence, the proof reduces to showing that $S_M^+$ is dense in $\mathcal{F}$, where
\begin{equation*}
\mathcal{F}:=\Big\{g: g=h\circ T-h \text{ for some } h\in L^\infty \Big\}.
\end{equation*}
Let $\epsilon>0$, $M\in \mathbb N$, and $ g=(h\circ T-h)\in \mathcal{F} $ be arbitrary.
 Assume that $0<\epsilon<\frac{1}{M}$ and $\|h\|_\infty=K$.
\begin{lem}{}{}
Let $E, \epsilon, K, M, (N_i)$ be as in the above proposition. Then there exists a set $E$ such that $\mu(E)<\frac{\epsilon}{6K}$ and a set $F$ with $\mu(F)=1$ such that for every $x\in F$, there exists $i\geq M$ such that $T^{N_i+1}x\in E$.
\end{lem}
\begin{proof}
By deleting some terms if needed, we can assume that $(N_i)$ satisfies $\frac{N_{i+1}}{N_i}>2$. Then by \Cref{lem:sso}, there exists a set $E$ with $\mu(E)<\frac{\epsilon}{6K}$ such that 
\begin{equation*}
\limsup_i \dfrac{\setone_E(T^{N_1+1} x)+\setone_E (T^{N_2+1}x)+\dots \setone_E(T^{N_i+1}x)}{i}=1
\end{equation*} holds
for almost every $x\in X$. This gives our desired conclusion.
\end{proof}
Now, we want to construct a function $f$ such that $\|f-g\|_1<\epsilon$ and $f\in S_M^+$.\\
Define $f=\tilde{h}\circ T-\tilde{h}$ where $\tilde{h}=h+3K \setone_E$.
Observe that
\begin{align}
\int |h-\tilde{h}|&=\int_{E} |h-\tilde{h}|\leq 3K\frac{\epsilon}{6K}=\frac{\epsilon}{2}. \\
\text{Similarly,}
\int |h\circ T-\tilde{h}\circ T|&=\int_{{T^{-1}E}} |h-\tilde{h}|\leq 3K\frac{\epsilon}{6K}=\frac{\epsilon}{2}.
\text{ Hence }
\|f-g\|_1&<\epsilon.
\end{align}
To check that $f\in S_M^+$, let $x\in (E^c\cap F)$. Then
\begin{align}
\setA_{[N_i]} f(T^n x)&=  \dfrac{\tilde{h} (T^{N_i+1}x)-\tilde{h}(x)}{N_{i}}\\
&=\dfrac{\Big(  h(T^{N_i+1}x)-h(x)\Big)}{N_i}+3K \dfrac{\setone_E(T^{N_i+1}x)-\setone_E(x)}{N_i}
\end{align}
 Since $x\in E^c$, we get
 \begin{align}
\setA_{[N_i]} f(T^n x)=\dfrac{\Big( h(T^{N_i+1}x)-h(x)\Big)+3K \setone_{E}(T^{N_i+1}x)}{N_i}\label{eq:115}
\end{align}
By the above lemma, there exists $i_0>M$ such that $T^{N_{i_0}+1}x\in E$ i.e. $\setone_E(T^{N_{i_0}+1}x)=1$. We want to show that for this $i_0$, the expression in \eqref{eq:115} is positive. Since $\|h\|_\infty=K,$ we must have $ \Big( h(T^{N_{i_0}+1}x)-h(x)\Big)\leq 2K $. This implies that the expression in \eqref{eq:115} is at least $\dfrac{K}{N_i}$ and hence positive. Thus, we have
\begin{equation*}
\mu\Big\{x:\setA_{ [N_i]}f(T^n x)>0 \text{ for some } i\geq M\Big\}\geq \mu(E^c)\geq 1-\epsilon\geq 1-\frac{1}{M}.
\end{equation*}
Hence, we conclude that $f\in S_M^+$, thereby finishing the proof of the theorem.
\end{proof} 
\end{proof}

The following theorem provides a sufficient condition for non-monotonicity of the averages in terms of the complete recurrence property.
\begin{thm}\label{thm:6} 
Let $T$ be a totally ergodic transformation and $(N_i)$ be an increasing sequence of positive integers that admits a subsequence $(N_{i_j})$ such that
\begin{enumerate}
    \item $(N_{i_j+1}-N_{i_j})=K$ for some fixed constant $K$.
    \item $(N_{i_j})$ satisfies the complete recurrence property.
\end{enumerate} 
Then for any non-constant, integrable $f$, $\setA_{[N_{i}]}f(T^{n}x)$ is non-monotone.
\end{thm}
\begin{proof}
Without loss of generality, let us assume that $f\in L^1_0$, that is, $\int f=0$.\\
One can check that
\begin{align*}
\setA_{[N_{i_j+1}]}f(T^n x)-\setA_{[N_{i_j}]}f(T^n x)&=\Big(\setA_{[N_{i_j}+1,N_{i_j+1}]}f(T^n x)-\setA_{[N_{i_j}]}f(T^n x)\Big)\dfrac{(N_{i_j+1}-N_{i_j})}{N_{i_j+1}}.
\end{align*}
Since $\setA_{[N_{i_j}]}f(T^n x)\to 0$ for almost every $x$, it will be sufficient to show that there exists $\delta>0$ such that for a.e. $x$, $\setA_{[N_{i_j}+1,N_{i_j+1}]}f(T^n x)>\delta$ for infinitely many $j$ and $\setA_{[N_{i_l}+1,N_{i_l+1}]}f(T^n x)<-\delta$ for infinitely many $l$.\\
Define
\begin{equation*}
g(x)=\setA_{[K]}f(T^n x).
 \end{equation*}
Since  $(N_{i_j+1}-N_{i_j})=K$ for every $j$,
\begin{equation*}
\setA_{[N_{i_j}+1,N_{i_j+1}]}f(T^n x)= g(T^{N_{i_j}}x).
\end{equation*}
This suggests that we should analyze the function $g$.\\
Since $f$ is non-constant and $T$ is totally ergodic, by \Cref{cor:nonconstant}, $g$ is also a non-constant function. Since  $f\in L^1_0 $, we must have $g\in L^1_0$. Hence, there exists $\delta>0$ such that both $\mu(S_0^+)$ and $\mu(S_0^-)$ are positive, where
\begin{align*}
S_0^+= \{x:g(x)>\delta\},\text{ and }\  S_0^-= \{x:g(x)<-\delta\}.
\end{align*}
Since $(N_{i_j})$ satisfies the complete recurrence property, for a.e. $x$, there are infinitely many $j$ and $l$ such that $g(T^{N_{i_j}}x)>\delta$ and $g(T^{N_{i_l}}x)<-\delta$.
This finishes the proof.
\end{proof}
\subsection{Irrational rotation case}
Let us prove some corollaries of \Cref{thm:6}.
\begin{cor}
If $(X,\mathcal{B},\mu,T)$ is a totally ergodic transformation and $(N_i)$ is a sequence of positive integers that contains an infinite arithmetic progression, then for any non-constant, integrable $f$ , $\setA_{[N_{i}]}f(T^{n}x)$ is non-monotone.
\end{cor}
The proof is immediate from Theorem~\ref{thm:6}.

For the next corollaries, we need the following lemmas.
\begin{lem}\label{lem:dense} 
Let $E\subset \setT$ such that $\lambda(E)>0$ and $(x_j)$ is a sequence that is dense in $\setT$, then $\lambda\Big(\cup_{j\in \setN}(x_j+E)\Big)=1$.  
\end{lem}
\begin{proof}
Denote $F= \cup_j (E+x_j)$.  If we had  $\lambda(F)<1$, then $\lambda(F^c)>0$.  Let $I$ be an interval in which the density of $E$ is greater than $1/2$, that is, $\lambda(E\cap I)>1/2\lambda(I)$.  Let $J$ be an interval in which the density of $F^c$ is greater than $1/2$.  By the Lebesgue density theorem we can assume that $I$ and $J$ have the same length.  Because of the density of $x_j$, there is a subsequence $j_k$ so that $\lim_k (I+x_{j_k})=J$.  Let $G=\lim_k (E\cap I)+x_{j_k}$.  Then $\lambda(G\cap J)>1/2\lambda(J)$ and also $\lambda(F^c\cap J)>1/2\lambda(J)$.  But $\lambda(F^c\cap G\cap J)=0$, an impossibility. 
\end{proof}
\begin{cor}{}{}
Let $p_i$ denote the $i$-th prime. There exists a set $F\subset \setT$ with $\lambda(F)=1$ that satisfies the following:\\
For each $\alpha\in F$ and non-constant $f\in L^1$, $\setA_{ [p_i]}f(x+ n\alpha)$ is non-monotone.
\end{cor}
\begin{proof}
By a famous theorem of Zhang~\cite{Zhang}, we know that there exists a constant $k$ and a subsequence $(p_{i_j})$ of $(p_i)$ such that $p_{i_j+1}-p_{i_j}=k$. By Weyl's theorem, there exists a set $F$ with $\lambda(F)=1$ such that $(p_{i_j}\alpha)_j$ mod 1 is dense in $\setT$. Now, the corollary follows from \Cref{lem:dense} and \Cref{thm:6}.
\end{proof}
\begin{cor}\label{cor:syndetic} 
Let $(N_i)$ be an increasing sequence such that $\max_{i} (N_{i+1}-N_i)<+\infty$. Then there exists a set $F$ with $\lambda(F)=1$ that satisfies the following:\\
For each $\alpha\in F$ and non-constant $f\in L^1$, $\setA_{ [N_i]}f(x+ n\alpha)$ is non-monotone.
\end{cor}
The proof is similar to the previous corollary.

\medskip

Next, we will prove \Cref{thm:5}, which strengthens \Cref{cor:syndetic}. We require the following lemma for proving the theorem.

\begin{lem}\label{lem:Rosenblatt} 
 If a sequence $(N_i)$ of positive integers has positive lower density, then for all but countably many $\alpha$, $(N_i \alpha)_i$ (mod 1) is dense in $\mathbb{T}$.
\end{lem}
\begin{proof}
This was (independently) proved by Rosenblatt~\cite{Ros_Norm} and Boshernitzan ~\cite[Proposition 8.1]{BOS2}.
\end{proof}
\begin{proof}[Proof of \Cref{thm:5}]
Let $M$ be the maximum gap of $(N_i)$. For each $j\in [M]$, construct the set $S_j$ as follows:\\
If for any $i\in \setN$, $N_{2i}-N_{2i-1}=j$, then $N_{2i-1}\in S_j$.

Clearly, $S= \displaystyle \cdot\hspace{-13pt}\bigcup_{j\in [M]}\Big(S_j\cup (S_j+j)\Big)$. For $N_l\in S_j$ and $\alpha\in \setT$, we have
\begin{align}
\setA_{[N_{l+1}]}f(x+ n\alpha)-&\setA_{[N_{l}]}f(x+ n\alpha)\\
&=\Big(\setA_{[N_{l}+1,N_{l+1}]}f(x+ n\alpha)-\setA_{[N_{l}]}f(x+ n\alpha)\Big)\dfrac{(N_{l+1}-N_{l})}{N_{l+1}}\\
&=\Big(\setA_{[j]}f\big(x+({N_l+n})\alpha\big)-\setA_{[N_{l}]}f(x+ n\alpha)\Big)\dfrac{j}{N_{l+1}}.\label{eq:delta}
\end{align}

 For an $M$-tuple of positive integers $\mathsf{k}=(k_1,k_2,\dots,k_M)$ and $j\in [M]$, we consider 
\begin{align*}
\text{ the shifted sets 
 }S^{\mathsf{k}}_j=S_j+k_j
\text{ and their union  } S^{\mathsf{k}}= \cup_{j\in [M]}S^\mathsf{k}_j.
\end{align*}
Clearly, for every $\mathsf{k}$, $S^{\mathsf{k}}$ is a set of bounded gaps, and hence, it has positive lower density. But this implies, by \Cref{lem:Rosenblatt}, that $E^\mathsf{k}$ is countable, where 
\begin{align}
E^\mathsf{k}&=E(S^\mathsf{k})=\big\{\alpha\in \setT: \{\alpha s: s \in S^\mathsf{k}\} \text{ is not dense} \in \setT\big\}.
\end{align}
Hence, $E=\displaystyle\bigcup_{\mathsf{k}\in {\setN}^M} E^\mathsf{k}$ is also countable. Now, we claim that  if $\alpha\in \setT\setminus E$ and $f$ is a non-constant, integrable function, then for a.e. $x$, the averages $\setA_{ [N_i]}f(x+ n\alpha)$ are non-monotone.\\
For every $j\in [M]$, consider the function \(g_j(x)=\setA_{[j]}f(x+ n\alpha).\)

We know that for almost every $x$, $\setA_{[N_{l}]}f(x+ n\alpha) \to 0$ as $l\to \infty$. We will show that there is a $\delta>0$ such that for a.e. $x$, there are infinitely many $N_l\in S_j$ for some $j\in [M]$ satisfying $g_j(x+N_l \alpha)>\delta$ and infinitely many (different) $N_l\in S_w$ for some $w\in [M]$ satisfying $g_w(x+N_l \alpha)<-\delta$. If we can show this, then the expression in \eqref{eq:delta} will be $>0$ infinitely often and $<0$ infinitely often, which will finish the proof of the claim. We will show here the `$>$' part only (the other part is similar).

By using the total ergodicity of irrational rotation and \Cref{cor:nonconstant}, we can choose $\delta>0$ small enough so that for every $j\in [M]$, $\lambda(B_{j,0}^+)$ is positive, where
 \begin{equation*}
B_{j,0}^+= \{x:g_j(x)>\delta\}.
 \end{equation*}

 Again, by using the ergodicity of irrational rotation, we can find $k_1, k_2,\dots, k_M$ such that $\lambda(B)>0$ where 
 \begin{equation*}
B= \displaystyle\bigcap_{j\in [M]}\Big(k_j\alpha+B^{+}_{j,0}\Big).
 \end{equation*}

Fix $\mathsf{k}=(k_1,k_2,\dots,k_M)$. Let $(s^\mathsf{k}_n)_n$ be the sequence obtained by arranging the terms of  $S^\mathsf{k}$ in an increasing order.

Since $\alpha\in \setT\setminus E$, $(\alpha s^\mathsf{k}_n)_n$ is dense in $\setT$. Hence, by \Cref{lem:dense}, we have
 \begin{align*}
 \lambda\Big(\bigcup_{n\in \setN}(-\alpha s^\mathsf{k}_n+B\Big)=1.
 \end{align*}
 Hence, for a.e. $x\in \setT$ and every large $L\in \setN$, there exists $n\geq L$ such that $x\in (-\alpha s^\mathsf{k}_n+B)$ for some $n$. Let $s^\mathsf{k}_n\in S^{\mathsf{k}}_{j}$ for some $j\in [M]$, so $s^\mathsf{k}_n=s_{j}+k_{j}$ for some $s_j\in S_j$. This implies that $x\in -(s_{j}+k_{j})\alpha+B\subset -(s_{j}+k_{j})\alpha+( k_j\alpha+B^{+}_{j,0})= -s_{j}\alpha+B^{+}_{j,0}$. Hence, we get $g_{j}(x+ s_{j} \alpha)>\delta$, which is our desired conclusion.
\end{proof}
 \begin{rem}
\begin{enumerate}
  \item In view of \Cref{thm:6} and \Cref{lem:dense}, for an easier proof one would try to find a subsequence $(N_{i_j})$ of $(N_i)$ such that $(N_{i_j})$ has positive lower density and for some fixed constant $K$, it satisfies $N_{i_j+1}-N_{i_j}=K$ for every $j\in \setN$. However, in general, such a subsequence may not exist. For example, take a large block of 2-gap numbers, a large block of 3-gap numbers, another large block of 2-gap numbers, another large block of 3-gap numbers, and so on. While it is always possible to find a subsequence $(N_{i_j})$ of $(N_i)$ that has positive upper density and that satisfies $N_{i_j+1}-N_{i_j}=K$ for every $j\in \setN$ for some fixed constant $K$, \Cref{lem:dense} does not hold for a sequence of positive upper density (see \cite{BOS2} for a counterexample).
  
\item The conclusion of Theorem 1.9 may not hold for all $\alpha$. Indeed, take $(N_i)$ to be the increasing sequence which is obtained by arranging those $n\in \mathbb N$ which satisfy $n \alpha\in [\frac{1}{4},\frac{1}{2}]$ (mod $1$) and 
\begin{equation*}
 f(x):=\setone_{E}(x)-\setone_{E}(x+\alpha) \text{ where } E=\Big[\frac{1}{2},1\Big].  
\end{equation*}
With the above choice of $N_i$ and $\alpha$, one can check that for any $x\in [\frac{1}{4},\frac{1}{2}]$, we have $\setA_{ [0,N_{i+1})}f(x+ n\alpha)-\setA_{ [0,N_i)}f(x+ n\alpha)=\frac{1}{N_i}-\frac{1}{N_{i+1}}\geq 0$. Hence, there is monotonicity for a set of positive measure.  
\end{enumerate}
 \end{rem}
 \begin{prob}{}{}
Let $(X,\mathcal{B},\mu,T)$ be an ergodic system and $(N_i)$ be an increasing sequence of positive integers. Can we characterize those functions that fluctuate around their mean infinitely often?
 \end{prob}

\subsection{Some quantitative results}

In \Cref{prop:updowngap} we have seen a quantitative result between consecutive averages. Now we will investigate when the gap term $\frac dN$ can be made larger.   Take $w(N)$ that increases to $\infty$ as $N\to \infty$.  For a mean-zero $f\in L^1$, we want to see  what can be said about having $\setA_{[N+1]}f(T^n x) > \frac {w(N)}N + \setA_{[N]}f(T^n x)$ for a.e. $x$ and for infinitely many $N$. This condition clearly forces $w(N) = o(N)$ because a.e. $\setA_{[N]}f(T^n x) \to 0$ as $N\to \infty$.  

By our computation above, this inequality basically means that  $f(T^{N+1}x) > 2w(N)$ for infinitely many $N$.   So we cannot have a gap like this for $f\in L^\infty$.    But if $f\in L^1\setminus L^\infty$, then it is possible to have this occur. Similarly, at the same time, $\setA_{[N+1]}f(T^n x) < -\frac {w(N)}N + \setA_{[N]}f(T^n x)$ infinitely often occurs if $f(T^{N+1}x) \le -2w(N)$ for infinitely many $N$.  For unbounded $f\in L^1$ this is possible.  

The level $w(N)$ for which this can occur depends on $f$ and $T$.  For example, suppose we have $w(N) = o(N)$.  A related question is when can we have
$$\sum\limits_{N\in \setN} \mu\{x: f(x) > w(N)\}  = \infty?$$  It is likely that this occurs for a dense $G_\delta$ of functions in $L^1$.
But then also suppose we can take a  map $T$ such that $f\circ T^n$ are IID.  By the Borel-Cantelli Lemma, this will allow us to conclude that for a.e. $x$, $f(T^{N+1}x) > w(N)$ occurs for infinitely many $N$.  Finding such a $T$ may not work for $f$ on $X$, but we can certainly get many instances of the above by changing our measure space $X$ to $\prod\limits_{n\in \setZ}\{0,1\}$, the usual product measure $\prod\limits_{n\in \mathbb Z} \mu$, where $\mu(\{0\}) = \mu(\{1\}) = 1/2$, and letting $T$ be the shift map.  Then we take any $f\in L^1$ as above and replace it by $f_0 = f \circ \pi_0$ where $\pi_0$ is the projection of  $\prod\limits_{n\in \setZ}\{0,1\}$ onto the coordinate with $n=0$.  

The conclusion above is that when $w(N)$ increases to $\infty$ but $w(N)/N \to 0$ as $N\to \infty$, then there are many functions $f$ and ergodic maps $T$ such that for a.e. $x$, both $\setA_{[N+1]}f(T^n x) > \frac {w(N)}N + \setA_{[N]}f(T^n x)$ and 
$\setA_{[N+1]}f(T^n x) < -\frac {w(N)}N + \setA_{[N]}f(T^n x)$.  Indeed, this is the typical case.  We take here $\mathcal M$ to be the invertible, measure-preserving maps of $(X,\mathcal{B},\mu)$ given the usual weak topology.  

\begin{prop}\label{prop:shrink}   Suppose $w(N)$ is increasing to $\infty$ and $w(N)/N\to 0$ as $N\to \infty$.  Then there is a dense $G_\delta$ set $\mathcal G \subset  L^1$ such that for all $f\in \mathcal G$, there is a dense $G_\delta$ set $M(f)\subset \mathcal M$ of ergodic maps $T$ such that 
for a.e. $x$, we have both $\setA_{[N+1]}f(T^n x) >\frac {w(N)}N + \setA_{[N]}f(T^n x)$ infinitely often and
$\setA_{[N+1]}f(T^n x) < -\frac {w(N)}N + \setA_{[N]}f(T^n x)$ infinitely often.  
\end{prop}
\begin{proof} 
We know that $f$ is integrable on a probability space if and only if $$\sum\limits_{N\in \setN} \mu\big\{|f|\ge N\big\} < \infty.$$

It is not hard to see this is the best possible since if $\rho_N\to 0$, then for a dense $G_\delta$ set of functions in $L^1$, we have $\displaystyle\sum_{N\in \setN} \mu\big\{|f| > \rho_N \cdot N\big\}  = \infty$.  So if $w_N > 0$ and $w_N/N \to 0$, taking this as $\rho_N$, we see that for a dense $G_\delta$ set of functions $\displaystyle\sum_{N\in \setN} \mu\big\{|f| > w_N\big\} = \infty$.

A little more work will show even more, as stated in the proof: for a dense $G_\delta$ set of functions, both $\displaystyle\sum_{N\in \setN} \mu\big\{f > 2w_N\big\} =\infty$ and $\displaystyle\sum_{N\in \setN} \mu\big\{f < - 2w_N\big\} = \infty$.

We take such a class $\mathcal F$ as above in $L^1$. We will use the argument of \cite[Theorem 4.3]{RR}. However, here we will use $B_N = \{f > 2w_N\}$ and $B_N =\{f < - 2w_N\} $, respectively, and use $T$ to denote the transformation, not $\tau$.   

 By Theorem 4.3 of \cite{RR} , for all the $f\in \mathcal F$, we have a dense $G_\delta$ set of maps $T$ such that for a.e. $x$, $T^Nx \in B_N$ infinitely often, for both types of $B_N$.

Now we use the analysis before \Cref{prop:shrink} to both of the following inequalities: $$\setA_{[N+1]}f(T^n x) > \frac{w_N}{N} + \setA_{[N]}f(T^n x) \text{ and } \setA_{[N+1]}f(T^n x) < -\frac{w_N}{N} + \setA_{[N]}f(T^n x).$$   Because we have used $2w_N$ in defining the $B_N$, we have a little more room in the inequalities.  So this gives us the divergent series results as stated in \cref{prop:shrink}.
\end{proof}

 \begin{rem}\label{LEVELS}

It is worth noting that if $p > 1$, then $L^p$ is first category in $L^1$, and indeed the complement $L^1\backslash L^p$ is a dense $G_\delta$.  So if we ask for the generic behavior for $$\sum_{n\in \setN} \mu\{ |f| > w_n\} = \infty, \text{ and } w_n = n^{1/p},$$ we can use this.

However, for more general $w_n$ with $w_n/n \to 0$, something more is needed.  Probably, in this case there is always an increasing $\phi:[0,\infty)\to [0,\infty)$ such that $\phi(n) \ge w_n$ and yet still $\phi(n)/n \to 0$.   Then one argues that the $f\in L^1$ such that $\phi^{-1}(|f|)$ is not integrable is a generic set.  So generically we would have
\begin{align*}
\sum\limits_{n\in \setN} \mu\big\{|f| \ge w_n\big\} \ge \sum\limits_{n\in \setN} \mu\big\{|f| \ge \phi (n)\big\} = \sum\limits_{n\in \setN} \mu\big\{\phi^{-1}(|f|) \ge n\big\} = \infty.
\end{align*}

We use $\phi$ and consider an Orlicz space of functions where $\phi^{-1}(|f|)$ is integrable and show that it is first category in $L^1$, with a dense $G_\delta$ complement, just as is done in the special cases where $\phi(n) = n^{1/p}$.
 \end{rem}

Now, we will deduce some quantitative results about the non-monotonicity of subsequences of averages by using some strong sweeping-out results of moving averages.

Write
\begin{align}\label{eq:movingave}
\setA_{ [0,N_{i+1})}f(T^n x)&-\setA_{ [0,N_i)}f(T^n x)\\
    &= \Big(\setA_{ [N_i,N_{i+1})}f(T^n x) -\setA_{ [0,N_i)}f(T^n x)\Big)\dfrac{N_{i+1}-N_{i}}{N_{i+1}}.
\end{align}
If we let $l_i= N_{i+1}-N_i$, then we get
\begin{equation*}
 \setA_{ [N_i,N_{i+1})}f(T^n x)=\frac{1}{l_i}\sum_{j\in [0,l_i-1]}f (T^{N_i+j}x). 
\end{equation*}
We introduce a sequence of operators $M_i:L^1\to L^1$ as follows
\begin{equation}\label{eq:movingoperators}
M_if(x)=\frac{1}{l_i}\sum_{j\in [0,l_i-1]}f (T^{N_i+j}x).
\end{equation}

These operators are called the \emph{moving averages} for the system $(X,\mathcal{B},\mu,T)$ with parameters $(N_i, l_i)$.

\begin{prop}\label{prop:4} 
Assume that $T$ is ergodic, and $(M_i)$ satisfies the strong sweeping out property. Then there is a residual set $\beta\subset \mathcal{B}$ with the following property:\\
For any $E\in \beta$ there exists a constant $d= d(E, N_i)>0$ such that for a.e. $x$,
\begin{align*}
\setA_{ [0,N_{i+1})}\setone_E(T^n x)>\frac{d}{N_{i+1}}+\setA_{ [0,N_i)}\setone_E(T^n x) \text{ infinitely often,} \\
\text{ and for a.e. $x$, } \setA_{ [0,N_{i+1})}\setone_E(T^n x)<-\frac{d}{N_{i+1}}+\setA_{ [0,N_i)}\setone_E(T^n x) \text{  infinitely often.}
\end{align*}
 The constant $d$ can be taken to be $\frac{1}{2}\big(\min_{i}\{l_i\}\big)\cdot\big(\min\{\mu(E),1-\mu(E)\}\big).$
\end{prop}
\begin{proof}
Let $d$ be as mentioned in the proposition.
Since $T$ is ergodic, 
\begin{equation}\label{eq:15}
 \setA_{ [0,N_i)}\setone_E(T^n x)\to \mu(E) \text{ for a.e. } x. 
\end{equation}
Suppose $(M_i)$ satisfies the strong sweeping out. Then, by using \cite[Theorem 1.2]{delJuncoRosenblatt} and observing that $\setA_{ [N_i,N_{i+1})}f(T^n x)=M_i(x)$, we can find a residual set $\beta\subset \mathcal{B}$ such that for every $E\in \beta$,
\begin{equation}\label{eq:9}
    \limsup_{i} \setA_{ [N_i,N_{i+1})}\setone_E(T^n x)=1 \text{ for a.e. }x
\end{equation}
and 
\begin{equation}\label{eq:10}
\liminf_{i}\setA_{ [N_i,N_{i+1})}\setone_E(T^n x)=0\text{ for a.e. }x.
\end{equation}
Let $E\in \beta$. Let $\tilde{X}$ be the common set of full measure where \eqref{eq:15}, \eqref{eq:9} and \eqref{eq:10} hold.
Hence, for $x\in \tilde{X}$, there are two subsequences $(N_{i_j})=(N_{i_j})(x)$ and $(N_{i_k})=(N_{i_k})(x)$ such that
\begin{align}\label{eq:11}
\lim_{j} \setA_{[N_{i_j}+l_{i_j})}\setone_E(T^n x)=1
\end{align}
and 
\begin{align}\label{eq:12}
\lim_{k} \setA_{[N_{i_k},N_{i_k}+l_{i_k})}\setone_E(T^n x)=0.
\end{align}
Thus we have
\begin{align}
\lim_{j} \Big(\setA_{[N_{i_j},N_{i_j}+l_{i_j})}\setone_E(T^n x)-\setA_{N_{i_j}}\setone_E(T^n x)\Big)&= (1-\mu(E)) \text{ and } \label{eq:13}\\
\lim_{k} \Big(\setA_{[N_{i_k},N_{i_k}+l_{i_k})}\setone_E(T^n x)-\setA_{N_{i_k}}\setone_E(T^n x)\Big)&= (0-\mu(E)).\label{eq:14}
\end{align}
By \eqref{eq:movingave}, we have 
\begin{align*}
\setA_{N_{i_j+1}}\setone_E(T^n x)-\setA_{N_{i_j}}\setone_E(T^n x)&=\Big(\setA_{[N_{i_j},N_{i_j}+l_{i_j})}\setone_E(T^n x)-\setA_{N_{i_j}}\setone_E(T^n x)\Big)\dfrac{l_{i_j}}{N_{i_j+1}} 
\end{align*}
By using \eqref{eq:13}, we get
\begin{align*}
\Big(\setA_{N_{i_j+1}}\setone_E(T^n x)-\setA_{N_{i_j}}\setone_E(T^n x)\Big)
&\to (1-\mu(E))\dfrac{l_{i_j}}{N_{i_j+1}} \text{ as } j \to \infty.\end{align*}
Thus for large $j$,
\begin{align*}
\Big(\setA_{N_{i_j+1}}\setone_E(T^n x)-\setA_{N_{i_j}}\setone_E(T^n x)\Big)
&\geq \dfrac{d}{N_{i_j+1}}. \intertext{ Similarly, using \eqref{eq:movingave} and \eqref{eq:14}, we get}
\Big(\setA_{N_{i_k+1}}\setone_E(T^n x)-\setA_{N_{i_k}}\setone_E(T^n x)\Big) &\leq -\dfrac{d}{N_{i_k+1}} .
\end{align*}
This finishes the proof.
\end{proof}

We say a sequence $(N_i)$ is sublacunary if it satisfies the following
\begin{equation*}
\lim_{i\to \infty}\dfrac{N_{i+1}}{N_i}=1.
\end{equation*}

\begin{cor}{}{}
Assume that $T$ is an ergodic transformation on a non-atomic probability space $(X,\mathcal{B}, \mu)$ and $(N_i)$ is sublacunary. Then there is a residual set $\beta\subset \mathcal{B}$ with the following property:\\
For any $E\in \beta$ there exists a constant $d= d(E, N_i)>0$ such that for a.e. $x$, 
\begin{align*}
\setA_{ [0,N_{i+1})}\setone_E(T^n x)&>\frac{d}{N_{i+1}}+\setA_{ [0,N_i)}\setone_E(T^n x) \text{ infinitely often,}\\
 \text{ and for a.e. $x$, } \setA_{ [0,N_{i+1})}&\setone_E(T^n x)<-\frac{d}{N_{i+1}}+\setA_{ [0,N_i)}\setone_E(T^n x) \text{ infinitely often.}
\end{align*}
\end{cor}
\begin{proof}
This result follows from \cite[Theorem 2.5]{RW_Moving} and \Cref{prop:4}.
\end{proof}

\section{Fluctuation of convolution operators}\label{Derive}

Given $\phi \in L^1(\mathbb R)$ and a suitable function $f$, the 
convolution of $\phi$ and $f$ will be denoted by $\phi\star f (x)
= \int\limits_{-\infty}^{\infty} \phi(y) f(x - y) d\lambda(y)$.
The first basic general principle is this one.  

\begin{thm}   \label{thm:divergence} 
Let $B=L^p(\mathbb R)$ for some fixed  
$1 \le p \le \infty$, or $B=CB(\mathbb R)$.  
Assume that $(\phi_n)$ is a
proper approximate identity.   Let $(\epsilon_n)$ be a sequence of positive
  numbers with $\lim\limits_{n \to \infty} \epsilon_n = 0$. Then
there is a dense $G_{\delta}$ subset $\mathcal A$ 
of $B$ such that for all $f \in 
\mathcal A$, $\limsup\limits_{n \to \infty} \frac {|\phi_n\star f(x) -
  f(x)|}{\epsilon_n} = \infty$ a.e.
\end{thm}
\begin{proof}  We use Theorem 1.1 of \cite{delJuncoRosenblatt}. Because our approximate identity is
  normalized and asymptotically local, we can use this theorem, which
  is stated for probability spaces, in the context of convolutions on 
$\mathbb R$.  Indeed, what we show here is that for any interval $I =
  (a,b)$, there is a dense $G_{\delta}$ subset $\mathcal A_I$ in $B$
such that for $f \in 
\mathcal A_I$, $\sup\limits_{n \ge 1} \frac {|\phi_n\star f(x) -
  f(x)|}{\epsilon_n} = \infty$ a.e. on $I$.  Intersecting these sets over
all rational numbers $a$ and $b$ gives us a set 
$\mathcal A$ as in the statement
of the theorem.   

First, consider the case that $B = L^1(\mathbb R)$.  To use
Theorem 1.1 of \cite{delJuncoRosenblatt}, we need to show
that for any $K$ and $\epsilon < \frac 12$, there exists 
$f \in L^1(\mathbb R)$ with $\|f\|_1 \le 1$ such that $$\lambda\big\{x \in I: \sup\limits_{n \ge 1}\frac{ |\phi_n\star f-f|}{\epsilon_n} \ge K\big\} \ge \lambda(I) - \epsilon.$$  
For now also assume that $I = (0,1)$. 
Then, taking $M$ to be a whole number, 
let \linebreak
$~{f(x) = \frac 1{\lambda(I)} \frac {\sin (Mx)}{|\sin (Mx)|}\setone_I(x)}$ except
when $\sin (Mx) = 0$, and then let $f(x) = 1$.  
We see that $\|f\|_1 \le 1$.
Fix $N \ge 1$, and choose $M$ sufficiently large such that the rapid
fluctuation of $f$ causes $\|\phi_n\star f\|_{\infty} \le
\epsilon$ for $1 \le n \le N$. 
So for a fixed $n, 1 \le n \le N$, if $\epsilon_n \le \frac 1{2K}$,
this gives $\frac {|\phi_n \star f - f|}{\epsilon_n} \ge K$ a.e.
on $I$ because $|\phi_n\star f - f| \ge 1 - \epsilon \ge \frac 12$.
So certainly we can construct $f, \|f\|_1 \le 1$, with  $\lambda\big\{x \in I:
\sup\limits_{n \ge 1}\frac{ |\phi_n\star f-f|}{\epsilon_n} \ge K\big\} 
\ge \lambda(I) - \epsilon$. 
By translating and dilating 
the construction of $f$, we can achieve the
same result for any interval $I$, although the choice of $M$ will
need to take into account the size of $I$.    

The same argument works in any $L^p(\mathbb R)$ with some small
changes.  If $1 \le p <
\infty$, we take $f(x) = 
\frac 1{\lambda(I)^{1/p}} \frac {\sin (Mx)}{|\sin (Mx)|}\setone_I(x)$, and proceed
as above.  For
$L^\infty (\mathbb R)$ the normalizing factor is taken to be $1$,
not $\frac 1{\lambda(I)^{1/p}}$.  A little
more care in the construction of $f$ is needed when $B = CB(\mathbb
R)$, since we must make sure that the function is continuous.  In this
case, again letting $I = (0,1)$ at the outset, take 
$g \in CB(\mathbb R)$ to be a
piecewise linear function such that $g = 1$ on $(0,\frac 12 -
\epsilon)$, $g = -1$ on $(\frac 12 + \epsilon,1)$, and $g = 0$ otherwise.  
Then let $f(x) = g(Mx)$ for all $x \in \mathbb R$.
As we decrease
$\epsilon$ and increase $M$, the increased fluctuation in $f$ on $I$
will allow us to have simultaneously $|\phi_n \star f| \le \epsilon$ for
all $n, 1 \le n \le N$, and $|f| = 1$ on some $J \subset I$ with $\lambda(J)\ge 1 -
\epsilon$.  Again, since $\lim\limits_{n \to \infty} \epsilon_n = 0$,
this gives the desired estimate on the function on $(0,1)$.  By
translating and dilating the construction of $f$, we can achieve the
the same result for any interval $I$.
\end{proof}

 \begin{rem} 
{\bf a)} Theorem~\ref{thm:divergence} is really meant to address the
case where $(\phi_n\star f)$ converges to $f$ a.e. for all $f \in B$.
If not, having a ratio divergence does not have the same
value.  But even in this situation, the fact that the divergence
could occur for different functions because of different
very sparse subsequences of $(\phi_n)$ suggests that even in this case
the theorem tells us something worthwhile.

\noindent {\bf b)} One can possibly think that dividing by a sequence
of constants $\epsilon_n$ that tend to $0$ is too strict, but rather
one should look for a sequence of positive functions $(\epsilon_n(x))$ that
converges to $0$ a.e. to gauge the rate of convergence of the convolutions. 
However, if $\lim\limits_{n \to \infty} \epsilon_n(x) = 0$ a.e.,
then for any $\epsilon > 0$ and any bounded interval $I$, there
is $J\subset I$ such that $\lambda(I\backslash J) < \epsilon$ and
$\lim\limits_{n \to \infty} \epsilon_n(x) = 0$ uniformly on $J$.
Using this, the argument in Theorem~\ref{thm:divergence}
shows that there is a dense $G_{\delta}$ of functions $f$ in $B$ 
such that $\limsup\limits_{n \to \infty} \frac {|\phi_n\star f(x) -
  f(x)|}{\epsilon_n(x)} = \infty$ a.e. 

\noindent {\bf c)} We could also develop other types of measures of the
rate of convergence of $(\phi_n\star f)$ as was done in the ergodic
theory context in \cite[Theorem 1.1]{delJuncoRosenblatt} using the principle 
formulated there. For example,
even in the best situation, where $\phi_n = n\setone_{[-1/n,0]}, n \ge 1$,
as long as a sequence of positive numbers $(w_n)$ has
$\sum\limits_{n\in\setN} w_n = \infty$, the series 
$\sum\limits_{n\in\setN} w_n (\phi_n\star f - f)$ will diverge almost
everywhere for a dense $G_{\delta}$ set of functions in $B$.  The
proof of this and related variations on this theme are reasonably 
straightforward adaptations of the ergodic theoretic results that are 
given in \cite{delJuncoRosenblatt}.  

\noindent {\bf d)} It is clear that there are non-trivial functions,
even continuous functions, such that we do not have any such
divergence.  Having a characterization of this, or at least some simple constructions, would be
worthwhile.
 \end{rem} 

With a little more attention to the construction, we can improve
Theorem~\ref{thm:divergence} to prove \Cref{thm:bothwaysae}.

\begin{proof}[Proof of \Cref{thm:bothwaysae}]
The mechanics of the proof are the same as for the proof of
Theorem~\ref{thm:divergence}.
We illustrate this first for $B = L^1(\mathbb R)$.   What we need to show
is that
for a fixed interval $I$, we have for all $K > 0$ and $\epsilon >
0$, there
exists some function $f \in L^1(\mathbb R)$ with $\|f\|_1 \le 1$ such that
$\lambda\{x \in I:
\sup\limits_{ n \ge 1} \frac {\phi_n \star f(x)-f(x)}{\epsilon_n} \ge K\}
\ge \lambda(I) -
\epsilon$.  There is no harm in assuming that $I = (0,1)$.  Let $(r_m)$ be
the sequence
of Rademacher functions.  Let $E_m = \{x \in I: r_m(x) = -1\}$.  We will
be choosing a finite
subsequence $r_{m_1},\dots,r_{m_L}$.  The mutual independence of the
Rademacher functions
shows that $\lambda(\bigcup\limits_{l=1}^L E_{m_l}) \ge 1 - (\frac 12)^L$.
Take a set of positive numbers $\{a_1,\dots,a_L\}$.
Let $f = \sum\limits_{l\in[L]} a_lr_{m_l}$.  We will want to have $\|f\|_1
\le 1$, so
for now we just
fix $a_l > 0$ so that $\sum\limits_{l\in[L]} a_l \le 1$.  We consider some
finite
subsequence $(\phi_{n_l}:l =1,\dots,L)$ and 
a fixed $j=1,\dots,L$.  We now estimate
how large $\frac {\phi_{n_j} \star f - f}{\epsilon_{n_j}}$ can be made by using
the fact that
\begin{align*}
\label{eq:underest}
\frac {\phi_{n_j} \star f - f}{\epsilon_{n_j}} &\ge
 \frac {\phi_{n_j} \star (a_jr_{m_j}) - a_jr_{m_j}}{\epsilon_{n_j}}\\
&-\sum\limits_{l\in[j-1]} \frac {\phi_{n_j} 
\star (a_lr_{m_l}) -a_lr_{m_l}}{\epsilon_{n_j}} 
-\sum\limits_{l\in [j+1,L]} \frac {\phi_{n_j} \star (a_lr_{m_l}) -
a_lr_{m_l}}{\epsilon_{n_j}}.
\end{align*}

We make the term $\frac {\phi_{n_j} \star (a_jr_{m_j}) -
a_jr_{m_j}}{\epsilon_{n_j}}$ large
on $E_{m_j}$ just as in the proof 
of Theorem~\ref{thm:divergence} except that we
also have to
compensate for the scaling by $a_j$.  That is, we can fix $n_j$
large and then
choose $m_j$ so that $\phi_{n_j}\star (a_jr_{m_j}) - a_jr_{m_j} 
\ge \frac {a_j}2$ on $E_{m_j}$. 
So if we have originally chosen $n_j$ so that 
$\frac {a_j}{2\epsilon_{n_j}} \ge 2K$, we get
$\frac {\phi_{n_j} \star (a_jr_{m_j}) - a_jr_{m_j}}{\epsilon_{n_j}} \ge 2K$ on
$E_{n_j}$.
 This makes it possible to
choose $n_j$ to give us this estimate even though
$a_j$ may be very small.

Our construction of $(a_l: l =1,\dots, L), (n_l:l=1,\dots,L)$ and
$(m_l:l=1,\dots,L)$ will be inductive.
We begin with any fixed $a_1, 0 < a_1 < 1$, and then
choose $n_1$ and $m_1$ so that
$\frac {\phi_{n_1} \star (a_1r_{m_1}) - a_1r_{m_1}}{\epsilon_{n_1}} \ge 2K$ on
$E_{n_1}$.
Now suppose that $j=1,\dots,L$ is fixed and we have chosen $a_l, n_l$ and $m_l$
for
$l=1,\dots,j-1$.  We then choose $n_j$ large enough for the convolutions
$\phi_{n_j} \star (a_lr_{m_l})$ to be equal to 
$a_lr_{m_l}$ everywhere except for a set of
small measure
made up of intervals around the discontinuities of $r_{m_l}$ for all
$l=1,\dots,j-1$.
So we can arrange for
$\sum\limits_{l\in[j-1]} \frac {\phi_{n_j}\star (a_lr_{m_l}) -
a_lr_{m_l}}{\epsilon_{n_j}}$ to
be $0$ except for a set $F_j$ of measure as small as we like in $I$.  We
can also have chosen
$n_j$ and $m_j$ as discussed above so that we have
$\frac {\phi_{n_j}\star ( a_jr_{m_j}) - a_jr_{m_j}}{\epsilon_{n_j}} \ge 2K$ on
$E_{n_j}$.  So
this means that we can arrange that
$\frac {\phi_{n_j} \star (a_jr_{m_j}) - a_jr_{m_j}}{\epsilon_{n_j}}
-\sum\limits_{l\in[j-1]} \frac {\phi_{n_j}\star ( a_lr_{m_l}) -
a_lr_{m_l}}{\epsilon_{n_j}} \ge 2K$
on $\widetilde {E_{n_j}} = E_{n_j} \backslash F_j$. But now the sum
$\sum\limits_{l\in [j+1,L]} \frac {\phi_{n_j} \star (a_lr_{m_l}) -
a_lr_{m_l}}{\epsilon_{n_j}}$ 
comes from terms chosen later in the inductive construction, and it could
affect that
estimate that we have made already.  However, we do have the freedom to
place restrictions on  the values
of $a_l, l = j+1,\dots,L$ making them 
sufficiently small, without disturbing the
necessary fact that
$\sum\limits_{l\in [L]}a_l \le 1$, so we can arrange that
$\sum\limits_{l\in [j+1,L]} \frac {\phi_{n_j} \star (a_lr_{m_l}) -
a_lr_{m_l}}{\epsilon_{n_j}}$
is uniformly less than $\frac K2$ on $I$.  In summary, on $\widetilde
{E_{m_j}}$,
we have 
$\frac {\phi_{n_j} \star f - f}{\epsilon_{n_j}} \ge \frac {3K}2$.
In this inductive step, we also assume, of course, that we have
chosen $a_j$ so that $a_j > 0$ and
$\sum\limits_{l\in [j]} a_l < 1$.
This ends the inductive step.  We notice that at the $j$th stage
of the induction we are revising the restrictions
on $(a_l:l=j+1,\dots,L)$ that were imposed from earlier steps in
the induction to make the future choices of $a_l$ even smaller than
before.  
In any case,  we can choose the sets $F_{n_j}$ of sufficiently
small measure so that
 the function $f$ has  $ \|f\|_1
\le 1$
and $\lambda\big\{x \in I:
\sup\limits_{ n \ge 1} \frac {\phi_n \star f(x)  -f(x)}{\epsilon_n} \ge K\big\}
\ge \lambda(I) -
\epsilon$.

We now know that there is a dense
$G_{\delta}$
set $\mathcal A^+$ in $L^1(\mathbb R)$ such that for all $f \in \mathcal
B^+$ we have
$\limsup\limits_{n \to \infty} \frac {\phi_n \star f - f}{\epsilon_n} =
\infty$
almost everywhere.
By working with the sets $\{x \in I: r_m(x) = 1\}$, we could
similarly arrange
to have this companion lemma:
for all $K > 0$ and $\epsilon > 0$, there
exists some function $f \in L^1(\mathbb R)$ with $\|f\|_1 \le 1$ such that
$\lambda\{x \in I:
\inf\limits_{ n \ge 1} \frac {\phi_n \star f(x)  f(x)}{\epsilon_n} \le
-K\} \ge \lambda(I) -
\epsilon$. Then we can conclude that there is a
dense $G_{\delta}$
set $\mathcal A^-$
in $L^1(\mathbb R)$ such that for all $f \in \mathcal A^-$ we have
$\liminf\limits_{n \to \infty} \frac {\phi_n \star f - f}{\epsilon_n} =
-\infty$
almost everywhere.  The intersection $\mathcal A = \mathcal A^+ \cap
\mathcal A^-$
gives the desired dense $G_{\delta}$ set.

The case where $B = L^p(\mathbb R)$ is handled similarly except for the
normalization inherent in the choice of $(a_l:l=1,\dots,L)$.  The case of
$CB(\mathbb R)$
requires more care as before, because the functions we construct for our
lemmas must
be continuous.  But there is not much difficulty in handling this, in the
same
way that it was dealt with in Theorem~\ref{thm:divergence}, at the expense
of having a somewhat large set $F_{n_j}$ to deal with, albeit one that can be
chosen inductively
to have measure as small as we like at any stage of the construction.
\end{proof}

\begin{rem}\label{MINUS2}  
\begin{enumerate}
    \item 
There may be value in having the two separate Baire category arguments.  But again, we can use the first one to get the result.  Once we know that for a generic set $\mathcal{C}$, we have for $f\in \mathcal{C}$, 
$\limsup\limits_{n \to \infty} \frac {\phi_n
\star f(x)
-f(x)}{\epsilon_n} = \infty$ almost everywhere,
then we would have a generic set $\mathcal A = -\mathcal{C}$ such that for $f\in \mathcal A$, we have
$\limsup\limits_{n \to \infty} \frac {\phi_n
\star (-f(x)
-(-f)(x)}{\epsilon_n} = \infty$ almost everywhere.
But then, because in general we have $\limsup\limits_{n\to \infty} -\rho_n = -\liminf\limits_{n\to \infty} \rho_n$,
we would have for $f\in \mathcal A$,
$\liminf\limits_{n \to \infty} \frac {\phi_n
\star f(x)
-f(x)}{\epsilon_n} = -\infty$ almost everywhere.
So taking $\mathcal A\cap \mathcal{C}$, we would get a generic set on which we have the fluctuation around the limit that we want to demonstrate.

    \item
In the case that we are considering $B = CB(\mathbb R)$, it seems
reasonable to ask to improve Theorem~\ref{thm:divergence} by requiring
\emph{everywhere} divergence of the ratios.  This is possible, but not
without some technical difficulties.  In our upcoming work, we will explore this topic.
\end{enumerate}
\end{rem}

\section{Martingales}\label{Mart}

In this section, we will prove \Cref{thm:osc}.

\begin{proof}[Proof of \Cref{thm:osc}]
 We can prove the generic behavior occurs above in the following fashion.  
First, fix $l\ge 1$ and $N_0\ge 1$.  Then let $\mathcal O_{l,N_0}$ consist of all 
$G\in L^1[0,1]$ such that $\lambda\{x: E(G|\mathcal D_n)(x) > G(x) \,\text{for some}\,\, n, N_0 \le n \}> 1-\frac 1l$.
By using a similar argument as in \Cref{thm:generic fluctuation}, one can check that $\mathcal O_{l,N_0}$ is open in the $L^1$-norm topology.  Suppose we
show that it is dense in that topology too.  Then $\mathcal O^+ = \bigcap\limits_{l=1}^\infty
\bigcap\limits_{N_0=1}^\infty \mathcal O_{l,N_0}$ is a dense $G_\delta$ set.  Clearly, any $G\in \mathcal O^+$
has the property that for a.e. $x$, there are infinitely many $n$ with $E(G|\mathcal D_n)(x) > G(x)$.  But now
let $\mathcal O^- = \{G: -G \in \mathcal O^+\}$.  Then this is also a dense $G_\delta$ set, and
$G\in \mathcal O^-$
has the property that for a.e. $x$, there is some $n$ with $E(G|\mathcal D_n)(x) < G(x)$.  
Hence, any function $G$ in the dense $G_\delta$ set $\mathcal O= \mathcal O^+ \cap \mathcal O^-$ will have the fluctuation
property that we claimed.
  
So what we want to prove is that each $\mathcal O_{l,N_0}$ is a dense set in the $L^1$-norm topology. 
So fix a function $G\in L^1[0,1]$.  Then, for any fixed $N_0$, we 
want to construct a function $F\in L^1[0,1]$ such that $\|G-F\|_1$ is 
small and such that the conditional expectations 
$E(F|\mathcal D_n)$
fluctuate around the limit $F$ on most of $[0, 1]$, once we take a large number $N$
and consider all $n, N_0\le n \le N$. 

Here is what we do. First, for 
any $\epsilon > 0$, we can find
$M$ sufficiently large and $F_0$ which is measurable with respect to $\mathcal D_M$
and such that  $\|G-F_0\|_1 < \epsilon$.  Now let $\delta > 0$.  
Then we consider the Rademacher functions $r_k$ with $k \ge 0$. Let
$\Sigma_N = \sum\limits_{k\in [0,N]} r_k$.  We take 
$F = F_0 + \delta \Sigma_N$.   For $n\ge M$, we have
$E(F|\mathcal D_n) = F_0 + \delta E(\Sigma_N|\mathcal D_n)$.  Now fix in addition
$N_0$ such that $M\le N_0$ without loss of generality.
For $N$ sufficiently large, we will show that $E(\Sigma_N|\mathcal D_n)$ fluctuates 
around its limit $\Sigma_N$ on most 
of $[0,1]$, with $N_0 \le n \le N$.  Since $M\le N_0\le n$, we have
$E(F|\mathcal D_n) = F_0 + \delta E(\Sigma_N|\mathcal D_n)$.  So we have
also $E(F|\mathcal D_n)$ fluctuating in the same manner around its limit $F$.
Since below we will be fixing $N$, even though it is large, we can now choose $\delta$ sufficiently
small that $\|G - F\|_1 \le 2\epsilon$.  This gives us the approximation of $G$ by
suitably fluctuating conditional expectations that we need to prove the Baire category theorem - that
the dyadic martingale is fluctuating generically.

Let us now look more closely at what the Rademacher functions do for us here.
We have $r_0 = \setone_{[0,1/2]} - \setone_{[1/2,1]}$, and $r_k(x) = r_0(\{2^k x\}) = r_0(2^kx\mod1)$ for
all $x \in [0, 1]$, and $k\ge 0$. So $r_k$ is measurable with respect to $\mathcal D_{k+1}$.
Also, for $1\le n \le k$, we have $E(r_k|\mathcal D_n) = 0$.   The Rademacher functions
are IID mean-zero functions taking the values of $\pm 1$.  So the functions $n\to 
\Sigma_n = \sum\limits_{k\in [0,n]} r_k$ are a model for the classical symmetric random walk on the integers.

So now consider
the conditional expectations $E(\Sigma_N|\mathcal D_n)$.  For $n\ge N+1$, we would
have all $r_k,0\le k\le N$ measurable with respect to $\mathcal D_n$, and hence
$E(\Sigma_N|\mathcal D_n) = \Sigma_N$.  Our goal is to show that for all 
$N_0$ and $\eta > 0$, there is $N\ge N_0$ large enough such
that for all $x \in [0, 1]$  except for a set of measure smaller than $\eta$, we have
\begin{equation}\label{above}
E(\Sigma_N|\mathcal D_n)(x) >  \Sigma_N(x) \,\text{for some} \, \,n, N_0 \le n \le N \text{ and }
\end{equation}
\begin{equation}\label{under}
E(\Sigma_N|\mathcal D_n)(x) <  \Sigma_N(x)  \,\text{for some (other)} \,n, N_0 \le n \le N.
\end{equation}

By the definition of $r_k$ and $\mathcal D_n$, we actually have
$E(\Sigma_N|\mathcal D_n) = \sum\limits_{k\in [0,n-1]} r_k$.  So we are asking for
fluctuation of $\sum\limits_{k\in [0,n-1]} r_k$ around $\sum\limits_{k\in [0,N]} r_k$
using values of $n, N_0\le n\le N$.  Subtracting $E(\Sigma_N|\mathcal D_n)(x)= \sum\limits_{k\in [0,n-1]} r_k(x)$ from
both sides of the inequalities in \eqref{above} and \eqref{under}
 we are asking for the fluctuation of
$\sum\limits_{k\in [n,N]} r_k$ around $0$.   
This is like
the fluctuation of the symmetric random walk around $0$.  Here what we need to 
use is that the terms $r_k$ are IID mean-zero functions taking the values of $\pm 1$,
and so for large enough $N$, even though we will be restricting $n, N_0\le n \le N$,
we have $\sum\limits_{k\in [n,N]} r_k$ fluctuating around $0$ on $[0,1]$ except for a set of measure less than $\eta$. \end{proof}
 \begin{rem}
 There is clearly an index, aka``time reversal", occurring here.  If we were dealing with
a classical random walk, we would be looking at where $\sum\limits_{k\in [1,n]} r_k$ fluctuates around $0$ a.e., but
because we are looking at conditional expectations, which are strongly related to Lebesgue differentiation, we are
looking instead at where $\sum\limits_{k\in [n,N]} r_k$ fluctuates on most of the measure space as $N\ge N_0$ gets large, 
if we take $n$ with $N_0 \le n \le N$.
 \end{rem}
 \begin{prob}{}
\begin{enumerate}
    \item Can we describe completely the functions for which
we do not have fluctuation of the dyadic martingales in the limit, on some set of positive measure? 
\item Is this
class actually much different from the class of functions where $E(G|\mathcal D_n)$ is not monotone i.e.
a.e. $E(G|\mathcal D_n) > E(G|\mathcal D_{n+1})$ infinitely often, and a.e.
$E(G|\mathcal D_n) < E(G|\mathcal D_{n+1})$?

\item When is $E_n G$ monotone without being eventually constant?
 \end{enumerate}
 \end{prob}

\section{Uniformly distributed sequences}\label{UD}

Next, we consider the non-monotonicity and fluctuation of averages of uniformly distributed sequences.  We generally can take a sequence $\mathsf x:= (x_n)$ in $[0,1]$ that is {\em uniformly distributed}.  That is, for every continuous $f:[0,1]\to \mathbb C$, the averages $\setA_{[N]}f(\mathsf x) = \frac {1}{N}\sum\limits_{n\in [N]} f(x_n) \to \int_0^1 f(t) \, dt$ as $N\to \infty$.  The classical example closest to ergodic averages is where we take $\theta\in \mathbb R$ irrational and let $x_n$ be given by the fractional parts  $\{n\theta\}$ for all $n\ge 1$.  As usual, if we let $Tx = \{\theta +x\}$ for $x\in [0,1]$, then $T$ is ergodic and $\setA_{[N]}f(T^n x) = \frac {1}{N}\sum\limits_{n\in [N]} f(T^nx) = \frac {1}{N}\sum\limits_{n\in [N]} f(\{n\theta + x\})$.  For this average, we could take $f\in L^1([0,1],\lambda)$ where $\lambda$ is the normalized Lebesgue measure on $[0,1]$.  However, when we restrict $\setA_{[N]}f(T^n x)$ to $x=0$, then we have to also restrict $f$, for example by assuming it is continuous, in order that we would have convergence of the averages $\setA_{[N]}f(\mathsf x) = \frac {1}{N}\sum\limits_{n\in [N]} f(\{n\theta\})$.  For simplicity, let $I = \int_0^1 f(t) \, dt$.

As in ergodic dynamical systems, we have some basic inequalities.  
Suppose $\setA_{[N+1]}f (\mathsf x)\le \setA_{[N]}f(\mathsf x)$.  Then 
\[\frac 1{N+1} f(x_{N+1}) \le \Big(\frac 1N - \frac 1{N+1}\Big) \sum\limits_{n\in [N]} f(x_n)= \frac 1{N(N+1)} \sum\limits_{n\in [N]} f(x_n).\] This means that $f(x_{N+1}) \le \frac {1}{N} \sum\limits_{n\in [N]} f(x_n) = \setA_{[N]}f(\mathsf x)$. Similarly, if $\setA_{[N+1]}f(\mathsf x)\ge \setA_{[N]}f(\mathsf x)$ then $f(x_{N+1}) \ge \setA_{[N]}f(\mathsf x)$.

These basic inequalities give the following analogue to \Cref{updown}

\begin{prop}\label{prop:updownud}  Suppose a continuous $f:[0,1]\to \mathbb R$ is not a constant function.   Then for infinitely many $N$ we have $\setA_{[N+1]}f(\mathsf x) > \setA_{[N]}f(\mathsf x)$, and \   for infinitely many $N$ we have $\setA_{[N+1]}f(\mathsf x) < \setA_{[N+1]}f(\mathsf x)$.  Hence,\    the averages $\setA_{[N]}f(\mathsf x) $ are not eventually monotonically increasing or eventually monotonically decreasing.
\end{prop}
\begin{proof}
Since $f$ is not a constant, there must be some $\delta > 0$ such that $f \ge \delta +I$ on some non-empty open interval $I$, and $f \le -\delta +I$  on some non-empty open interval $J$.  

Assume  $\setA_{[N+1]}f(\mathsf x) \le \setA_{[N]}f(\mathsf x)$ holds for all large enough values of $N$.  This means $f(x_{N+1}) \le \setA_{[N]}f(\mathsf x)$ for all large enough $N$.  But $\setA_{[N]}f(\mathsf x)\to I$ as $N\to \infty$.  Hence, for all large enough $N$, $\setA_{[N]}f(\mathsf x) \le \delta/2 + I$.  However, we have  $x_{N+1} \in  I$ for infinitely many $N$.  So $\delta + I \le f(x_{N+1})$ for infinitely many $N$.  If we put these inequalities together, we get values of $N$ such that $\delta + I \le f(x_{N+1}) \le \delta/2 +I$.  Then $\delta \le \delta/2$, which is not possible.

Similarly, assume   $\setA_{[N+1]}f(\mathsf x) \ge \setA_{[N]}f(\mathsf x)$ holds for all large enough values of $N$.  This means that $f(x_{N+1}) \ge \setA_{[N]}f(\mathsf x)$ for all large enough $N$.  But $\setA_{[N]}f(\mathsf x)\to I$ as $N\to \infty$.  Hence, for all large enough $N$, $\setA_{[N]}f(\mathsf x) \ge -\delta/2 + I$.  However, we have  $x_{N+1}\in  J$ for infinitely many $N$.  So $-\delta + I \ge f(x_{N+1})$ for infinitely many $N$.  If we put these inequalities together, we get values of $N$ such that $-\delta + I \ge f(x_{N+1}) \ge -\delta/2 +I$.  Then $-\delta \ge -\delta/2$, which is not possible.
\end{proof}

As in \Cref{prop:updowngap}, we can prove this more detailed result.

\begin{prop}\label{prop:updowngapud}    Suppose $f:[0,1]\to \mathbb R$ is not a constant function.  Then there exists $d > 0$ such that we have  $\setA_{[N+1]}f(\mathsf x) > \frac dN + \setA_{[N]}f(\mathsf x)$ infinitely often, and 
 $\setA_{[N+1]}f(\mathsf x) < -\frac dN + \setA_{[N]}f(\mathsf x)$ infinitely often.
\end{prop}

\Cref{prop:updownud} answers one possible question about how $\setA_{[N]}f(\mathsf x)$ behaves as $N\to \infty$.  But it does not address whether or not there are infinitely many sign changes for $\setA_{[N]}f(\mathsf x) - I$.   In fact, as we will see, infinitely many sign changes may or may not occur, depending on the function $f$ and the context of the dynamics.  It is not clear how to characterize when it does occur in general.

We consider the special case of $\mathsf x= \{x_n\}_n = \{n\theta\}_n$ for an irrational $\theta\in \mathbb R$.  In addition, we take the special case of a continuous coboundary $f$.  This is a continuous function $f(t)  = F(t) - F(\{t+\theta\})$ for some continuous $F:[0,1]\to \mathbb R$. Then
\begin{align*}
\setA_{[N]}f(\mathsf x)= \frac 1N\sum\limits_{n\in [N]} f(\{n\theta\}) =  
\frac 1N \Bigl (F(\{\theta\})- F(\{(N+1)\theta\})\Bigr ).
\end{align*}
While this converges to $0$ as $N\to \infty$, whether or not there are changes of sign depend on whether or not $F(\{\theta\})- F(\{(N+1)\theta\})$ changes sign infinitely often.  

As we argued for fluctuation in the Cesàro-average case, with suitable hypotheses on the transfer function we can argue for fluctuation in the case of averages of uniformly distributed sequences.

\begin{prop}\label{prop:osccobud}   Suppose  a continuous, real-valued function $f$ is a coboundary with the transfer function $F$ that is continuous, real-valued, and not a constant. 
\begin{enumerate}
    \item Suppose that $F(\{\theta\})$ is the maximum (or minimum) value of $F$.  Then the averages $\setA_{[N]}f(\mathsf x)$ converge to $0$ as $N\to \infty$, but $\setA_{[N]}f(\mathsf x) \ge 0$ (or $\setA_{[N]}f(\mathsf x) \le 0$) for all $N\ge 1$.
    \item Suppose that $F(\{\theta\})$ is not a maximum (or minimum) value for $F$.  Then the averages $\setA_{[N]}f(\mathsf x)$ converge to $0$ as $N\to \infty$ and $\setA_{[N]}f(\mathsf x)$ is infinitely often less (or greater) than $0$.
\end{enumerate}  
\end{prop}

 \begin{rem}
We have taken advantage of $f$ being a coboundary to get a simpler expression for $\setA_{[N]}f(\mathsf x)$.  This has allowed us to obtain \Cref{prop:osccobud}.  But then the following question is left unresolved.

 Suppose $f$ is continuous and $I = 0$.  Under what necessary and sufficient conditions do we have $\setA_{[N]}f(\mathsf x)$ infinitely often greater than $0$, and infinitely often less than $0$?  In the light of \Cref{thm:characterization} , we conjecture the following:
 \end{rem}
\begin{conj}
If for a continuous function $f$, $\setA_{[N]}f(\mathsf x)$ fails to fluctuate in a set of positive measure, then $f$ can be expressed as $f(t)=F(t)-F(\{t+\theta\})$, where $F$ is a real valued, non-constant function, and $\theta$ is a fixed real number.
\end{conj}

On the other hand, as in \Cref{thm:generic fluctuation}, we can easily prove this result.

\begin{prop}{}{} For a given irrational $\theta$, there is a dense $G_\delta$ set $\mathcal O$ in the real-valued functions in $C[0,1]$, with the uniform norm, such that for all $f\in \mathcal O$, we have\   $\setA_{[N]}f(\mathsf x) > \int_0^1 f(t) \, dt$ infinitely often, and $\setA_{[N]}f(\mathsf x) < \int_0^1 f(t) dt$ infinitely often.
\end{prop}

 \begin{rem}  As remarked before, we only need to get one of these inequalities to hold generically, and then we can argue that they both do, at least generically of course.
\qed
 \end{rem}

\section{General weighted averages in ergodic theory}\label{GenCET}

Here we take a dissipative sequence $\mu_N$ of probability measures on  $\mathbb Z$.  Our operators $\mu_N^Tf = \sum\limits_{n\in \mathbb Z} \mu_N(n) f\circ T^n$.   Now we ask when these operators always fluctuate.  Because we are not assuming there is pointwise a.e. convergence, or even perhaps $L^1$-norm convergence, there is no clear limit about which we could consider fluctuation.  However, in the case that there is norm convergence, fluctuation around the limit will be worthwhile to consider.

\begin{conj} Let $\mu_N^T$ be a dissipative sequence of operators such that $\mu_N^Tf \to \int f$ in $L^1$-norm for every $f\in L^1$. Then there is a dense $G_\delta$ set $\mathcal{O} \subset L^1$ such that for every $f\in \mathcal{O}$, the averages $\mu_N^Tf(x)$ fluctuate infinitely often around $\int f$ for almost every $x\in X$.
\end{conj}

\noindent The fluctuation behavior of the classical averages clearly depended on the fact that these are Ces\`aro averages.  It may be worthwhile to ask when similar behavior holds for other methods of averaging like the ones above.
\begin{exmp}
Consider an ergodic map $T$ that has $-1$ as an eigenvalue.  Let $f\in L^1$ be mean-zero eigenvalue for $-1$, so that $f+ f\circ T = 0$.  Hence, $\setA^{\frac{1}{2}}f(T x) := \frac 12 \Big(f (Tx) + f (T^2x\Big) = 0$.  But this does not remain so, indeed for even $N$, $\setA^{\frac{1}{2}}_{[N]}f(T^n x) = 0$ but for odd $N$, $\setA^{\frac{1}{2}}_{[N]}f(T^n x) = -\frac 1N f$.  So the averages converge to $0$, but not monotone, and they do not switch sign infinitely often.
\end{exmp} 

However, if we use a different averaging method we can get a very different result.  Take $\setA^b_{[N]}f =\frac {1}{2^N} \sum\limits_{n\in [N]} {N\choose n} f\circ T^n $ for all $N\ge 1$. These averages do converge in $L^1$-norm to $\int f \, d\lambda$ for all $f\in L^1$, although they do not converge a.e. in general. See ~\cite{RRR}.  And we certainly can have monotonicity since  we have $\setA^b_{[N]}f = 0$ for all $N$ when $f\circ T = -f$ as above.  Of course, we do not have fluctuation of the sign.
\qed

This example makes it clear that whether, and when, we have monotonicity of the averages depends on the method of averaging that we use.  Moreover, for which $f$ and $T$ we have fluctuation of the values about the mean will also depend on the averaging method used.  

On the other hand, we can get some similar results to what happens using the classical ergodic averages, at least when using more general Ces\`aro averages.  This is in line with what was done in Section~\ref{UD}, where general uniformly distributed sequences were used instead of just the classical sequence $(\{n\theta\}:n\ge 1)$ with $\theta \in \mathbb R$ irrational.

\begin{prop}{}{}  Consider an increasing sequence $(m_n:n\ge 1)$ in $\mathbb Z^+$.  Assume that for all the $f\in L^2$, averages $\setA_{[N]}f(T^{m_n}x) := \frac 1N\sum\limits_{n\in [N]} f(T^{m_n} x)$ converge to $\int f \, d\lambda$ a.e.  Then for a non-constant $f\in L^2$, for a.e. $x$, the averages $\setA_{[N+1]}f(T^{m_n}x) > \setA_{[N]}f(T^{m_n}x)$ infinitely often, and the averages $\setA_{[N+1]}f(T^{m_n}x) < \setA_{[N]}f(T^{m_n}x)$ infinitely often.
\end{prop}


\section*{Acknowledgement}
The authors would like to thank Terence Adams for his valuable input throughout the paper. We thank the referees for careful reading of the manuscript and useful remarks that led to the improvement of the presentation.

\printbibliography

\end{document}